%% file: slrqpg.tex
\documentclass[a4paper]{article}
\input{slrqpg_head}

\begin{document}
\input{slrqpg_abstract}

\section{Introduction}

\subsection{Motivation}

The aim of the present work is to investigate scaling limits for random maps of arbitrary genus. Recall that a map is a cellular embedding of a finite graph (possibly with multiple edges and loops) into a compact connected orientable surface without boundary, considered up to orientation-preserving homeomorphisms. By \textit{cellular}, we mean that the faces of the map---the connected components of the complement of edges---are all homeomorphic to discs. The genus of the map is defined as the genus of the surface into which it is embedded. For technical reasons, it will be convenient to deal with rooted maps, meaning that one of the half-edges---or oriented edges---is distinguished.

We will particularly focus on bipartite quadrangulations: a map is a quadrangulation if all its faces have degree $4$; it is bipartite if each vertex can be colored in black or white, in such a way that no edge links two vertices that have the same color. Although in genus $g=0$, all quadrangulations are bipartite, this is no longer true in positive genus $g \ge 1$.

A natural way to generate a large random bipartite quadrangulation of genus $g$ is to choose it uniformly at random from the set $\Q_n$ of all rooted bipartite quadrangulations of genus $g$ with~$n$ faces, and then consider the limit as $n$ goes to infinity. From this point of view, the planar case---that is $g=0$---has largely been studied for the last decade. Using bijective approaches developed by Cori and Vauquelin \cite{cori81planar} between planar quadrangulations and so-called \textit{well-labeled trees}, Chassaing and Schaeffer \cite{chassaing04rpl} exhibited a scaling limit for some functionals of a uniform random planar quadrangulation. They studied in particular the so-called \textit{profile} of the map, which records the number of vertices located at every possible distance from the root, as well as its radius, defined as the maximal distance from the root to a vertex. They showed that the distances in the map are of order $n^{1/4}$ and that these two objects, once the distances are rescaled by the factor $n^{-1/4}$, admit a limit in distribution.

Marckert and Mokkadem \cite{marckert06limit} addressed the problem of convergence of quadrangulations as a whole, considering them as metric spaces endowed with their graph distance. They constructed a limiting space and showed that the discrete spaces converged toward it in a certain sense. The natural question of convergence in the sense of the Gromov-Hausdorff topology \cite{gromov99msr} remained, however, open. It is believed that the scaling limit of a uniform random planar quadrangulation exists in that sense. An important step toward this result has been made by Le Gall \cite{legall07tss} who showed the tightness of the laws of these metric spaces, and that every possible limiting space---commonly called Brownian map, in reference to Marckert and Mokkadem's terminology---is in fact almost surely of Hausdorff dimension $4$. He also proved, together with Paulin \cite{legall08slb}, that every Brownian map is almost surely homeomorphic to the two-dimensional sphere. Miermont \cite{miermont08sphericity} later gave a variant proof of this fact.

In positive genus, Chapuy, Marcus, and Schaeffer \cite{chapuy07brm} extended the bijective approaches known for the planar case, leading Chapuy \cite{chapuy08sum} to establish the convergence of the rescaled profile of a uniform random bipartite quadrangulation of fixed genus. A different approach consists in using Boltzmann measures. The number of faces is then random: a quadrangulation is chosen with a probability proportional to a certain fixed weight raised to the power of its number of faces. Conditionally given the number of faces, a quadrangulation chosen according to this probability is then uniform. Miermont \cite{miermont09trm} showed the relative compactness of a family of these measures, adapted in the right scaling, as well as the uniqueness of typical geodesics in the limiting spaces.

The present work generalizes a part of the above results to any positive genus: we will show the tightness of the laws of rescaled uniform random bipartite quadrangulations of genus~$g$ with~$n$ faces in the sense of the Gromov-Hausdorff topology. These results may be seen as a conditioned version of some of Miermont's results appearing in \cite{miermont09trm}. We will also prove that the Hausdorff dimension of every possible limiting space is almost surely $4$.

\subsection{Main results}

We will work in fixed genus $g$. On the whole, we will not let it figure in the notations, in order to lighten them. As the case $g=0$ has already been studied, we suppose $g \ge 1$.

We use the classic formalism for maps, which we briefly remind here. For any map $\m$, we denote by $V(\m)$ and $E(\m)$ respectively its sets of vertices and edges. We also call $\vec E(\m)$ its set of half-edges. By convention, we will note $\e_*\in \vec E(\m)$ the root of $\m$. For any half-edge $\e$, we write $\bar\e$ its reverse---so that $E(\m) = \{ \{ \e,\bar\e \}\,:\, \e\in\vec E \}$---as well as $\e_-$ and $\e_+$ its origin and end. Finally, we say that $\ori E(\m) \subset \vec E(\m)$ is an orientation of the half-edges if for every edge $\{ \e,\bar\e \} \in E(\m)$ exactly one of $\e$ or $\bar\e$ belongs to $\ori E(\m)$.

Recall that the Gromov-Hausdorff distance between two compact metric spaces $(\S,\delta)$ and $(\S',\delta')$ is defined by
$$\dgh\lp(\S,\delta),(\S',\delta')\rp \de \inf \big\{ d_{Haus}\big(\varphi(\S),\varphi'(\S')\big)\!\big\},$$
where the infimum is taken over all embeddings $\varphi : \S \to \S''$ and $\varphi':\S'\to \S''$ of $\S$ and $\S'$ into the same metric space $(\S'', \delta'')$, and $d_{Haus}$ stands for the usual Hausdorff distance between compact subsets of $\S''$. This defines a metric on the set of isometry classes of compact metric spaces (\cite[Theorem 7.3.30]{burago01cmg}), making it a Polish space\footnote{This is a simple consequence of Gromov's compactness theorem \cite[Theorem 7.4.15]{burago01cmg}.}.

Any map $\m$ possesses a natural graph metric $d_\m$: for any $a,b\in V(\m)$, the distance $d_\m(a,b)$ is defined as the number of edges of any shortest path linking $a$ to $b$. Our main result is the following.

\begin{thm}\label{cvq}
Let $\q_n$ be uniformly distributed over the set $\Q_n$ of all bipartite quadrangulations of genus $g$ with $n$ faces. Then, from any increasing sequence of integers, we may extract a subsequence $(n_k)_{k\ge 0}$ such that there exists a random metric space $(\q_\infty,d_\infty)$ satisfying
$$\Big( V(\q_{n_k}),\frac 1 {\gamma n_k^{1/4}}\, d_{\q_{n_k}} \Big) \tolk (\q_\infty,d_\infty)$$
in the sense of the Gromov-Hausdorff topology, where
$$\gamma \de \lp\frac {8}9 \rp^{\frac 14}.$$

Moreover, the Hausdorff dimension of the limit space $(\q_\infty,d_\infty)$ is almost surely equal to~$4$, regardless of the choice of the sequence of integers.
\end{thm}

The limiting spaces $(\q_\infty,d_\infty)$ appearing in Theorem~\ref{cvq} are expected to have similar properties as in the case $g=0$. For instance, they are expected to have the same topology as the torus with~$g$ holes, and to possess the property of uniqueness of their geodesic paths. In an upcoming work, we will show that the topology is indeed that of the $g$-torus.

\bigskip

We call \textit{$g$-tree} a map of genus $g$ with only one face. This generalizes the notion of tree: note that a $0$-tree is merely a (plane) tree. In order to show Theorem~\ref{cvq}, we will code quadrangulations by $g$-trees via a bijection introduced by Chapuy, Marcus, and Schaeffer \cite{chapuy07brm}, which we expose in Section~\ref{seccms}. We then study the scaling limits of $g$-trees: we first decompose them in Section~\ref{decomp} and study their convergence in Section~\ref{seccv} and~\ref{secmf}. Finally, Section~\ref{secpro} is dedicated to the proof of Theorem~\ref{cvq}.

\bigskip

Along the way, we will recover an asymptotic expression, already known from \cite{chapuy07brm}, for the cardinality of the set $\Q_n$ of all rooted bipartite quadrangulations of genus~$g$ with~$n$ faces. Following \cite{chapuy07brm}, we call \textit{dominant scheme} a $g$-tree whose vertices all have degree exactly~$3$. We write $\Sg^*$ the (finite) set of all dominant schemes of genus~$g$. It is a well-known fact that there exists a constant $t_g$ (only depending on~$g$) such that $|\Q_n|\sim t_g\, n^{\frac 52 (g-1)}\,12^n$ (see for example \cite{bender91number, chapuy07brm, miermont09trm}). This constant plays an important part in enumeration of many classes of maps \cite{bender91number,gao93number} .

\begin{thm}[\cite{chapuy07brm}]\label{tgprop}
The following expression holds
\begin{equation}\label{tg}
t_g = \frac{3^g}{2^{11g-7} \, (6g-3) \Gamma \lp \frac{5g-3}2 \rp}\, \sum_{\s\in\Sg^*} \sum_{\lambda\in \O_\s} \prod_{i=1}^{4g-3} \frac 1 {d(\lambda,i)},
\end{equation}
where the second sum is taken over all $(4g-2)!$ orderings $\lambda$ of the vertices of a dominant scheme $\s\in \Sg^*$, i.e. bijections from $\ent 0 {4g-3}$ onto $V(\s)$, and
\begin{equation}\label{dlk}
d(\lambda, k) \de \lm \lb \e\in\vec E(\s),\ \lambda^{-1}_{\e^-} < k \le  \lambda^{-1}_{\e^+} \rb \rm.
\end{equation}
\end{thm}

As the proof of this expression is more technical, we postpone it to the last section. By convention, we will suppose that all the random variables we consider are defined on a common probability space $(\Omega,\mathcal F,\P)$.


\section{The Chapuy-Marcus-Schaeffer bijection}\label{seccms}

The first main tool we use consists in the Chapuy-Marcus-Schaeffer bijection \cite[Corollary~2 to Theorem~1]{chapuy07brm}, which allows us to code (rooted) quadrangulations by so-called well-labeled (rooted) $g$-trees.

It may be convenient to represent a $g$-tree $\t$ with $n$ edges by a $2n$-gon whose edges are pairwise identified (see Figure \ref{gon}). We note $\e_1\de \e_*$, $\e_2$, \dots, $\e_{2n}$ the half-edges of $\t$ sorted according to the clockwise order around this $2n$-gon. The half-edges are said to be sorted according to the \textbf{facial order} of $\t$. Informally, for $2 \le i \le 2n$, $\e_i$ is the ``first half-edge to the left after $\e_{i-1}$.'' We call \textbf{facial sequence} of $\t$ the sequence $\t(0)$, $\t(1)$, \dots, $\t(2n)$ defined by $\t(0)=\t(2n) = \e_1^- = \e_{2n}^+$ and for $1 \le i \le 2n-1$, $\t(i)=\e_i^+=\e_{i+1}^-$. Imagine a fly flying along the boundary of the unique face of $\t$. Let it start at time $0$ by following the root $\e_*$ and let it take one unit of time to follow each half-edge, then $\t(i)$ is the vertex where the fly is at time~$i$.

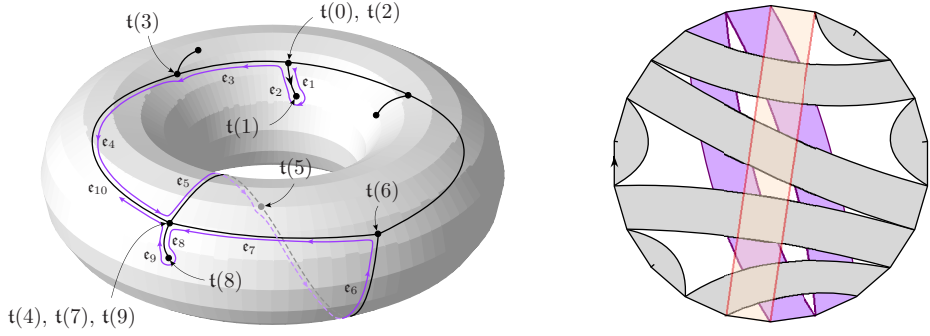
\tfig{gon}{On the left, the facial order and facial sequence of a $g$-tree. On the right, its representation as a polygon whose edges are pairwise identified.}{
\psfrag{0}[][][.62]{$\e_{10}$}
\psfrag{1}[][][.62]{$\e_1$}
\psfrag{2}[][][.62]{$\e_2$}
\psfrag{3}[][][.62]{$\e_3$}
\psfrag{4}[][][.62]{$\e_4$}
\psfrag{5}[][][.62]{$\e_5$}
\psfrag{6}[][][.62]{$\e_6$}
\psfrag{7}[][][.62]{$\e_7$}
\psfrag{8}[][][.62]{$\e_8$}
\psfrag{9}[][][.62]{$\e_9$}
\psfrag{w}[l][l][.8]{$\t(0)$, $\t(2)$}
\psfrag{x}[r][r][.8]{$\t(1)$}
\psfrag{c}[][][.8]{$\t(3)$}
\psfrag{v}[r][r][.8]{$\t(4)$, $\t(7)$, $\t(9)$}
\psfrag{h}[][][.8]{$\t(5)$}
\psfrag{n}[][][.8]{$\t(6)$}
\psfrag{j}[l][l][.8]{$\t(8)$}
}

Let $\t$ be a $g$-tree. The two vertices $u,v\in V(\t)$ are said to be \textbf{neighbors}, and we write $u\sim v$, if there is an edge linking them.

\begin{defi}
A \bf{well-labeled $g$-tree} is a pair $(\t,\l)$ where $\t$ is a $g$-tree and $\l:V(\t) \to \Z$ is a function (thereafter called \textbf{labeling function}) satisfying:
\begin{enumerate}[i.]
 \item $\l(\e_*^-) = 0$, where $\e_*$ is the root of $\t$,
 \item if $u \sim v$, then $|\l(u) - \l(v)| \le 1$.
\end{enumerate}
\end{defi}

We call $\T_n$ the set of all well-labeled $g$-trees with $n$ edges.

\bigskip

A \textbf{pointed quadrangulation} is a pair $(\q,v_\bullet)$ consisting in a quadrangulation $\q$ together with a distinguished vertex $v_\bullet \in V(\q)$. We call
$$\Qb_n \de \lb (\q,v_\bullet):\, \q\in \Q_n,\, v_\bullet \in V(\q) \rb$$
the set of all pointed bipartite quadrangulations of genus $g$ with $n$ faces.

\bigskip

The Chapuy-Marcus-Schaeffer bijection is a bijection between the sets $\T_n \times \{-1,+1\}$ and $\Qb_n$. As a result, because every quadrangulation $\q \in \Q_n$ has exactly $n+2-2g$ vertices, we obtain the relation
\begin{equation}\label{qt}
(n+2-2g)\,|\Q_n|=2 \, |\T_n|.
\end{equation}

Let us now briefly describe the mapping from $\T_n \times \{-1,+1\}$ onto $\Qb_n$. We refer the reader to \cite{chapuy07brm} for a more precise description. Let $(\t,\l) \in \T_n$ be a well-labeled $g$-tree with $n$ edges and $\eps\in \{-1,+1\}$. As above, we write $\t(0)$, $\t(1)$, \dots, $\t(2n)$ its facial sequence. The pointed quadrangulation $(\q,v_\bullet)$ corresponding to $((\t,\l),\eps)$ is then constructed as follows. First, shift all the labels in such a way that the minimal label is~$1$. Let us call $\tilde\l \de \l-\min\l +1$ this shifted labeling function. Then, add an extra vertex $v_\bullet$ carrying the label $\tilde\l(v_\bullet)\de 0$ inside the only face of $\t$. Finally, following the facial sequence, for every $0\le i \le 2n-1$, draw an arc---without crossing any edge of $\t$ or arc already drawn---between $\t(i)$ and its successor, defined as follows:
\begin{itemize}
	\item if $\tilde\l(\t(i))=1$, then its successor is the extra vertex $v_\bullet$,
	\item if $\tilde\l(\t(i))\ge 2$, then its successor is the first following vertex whose shifted label is $\tilde\l(\t(i)) - 1$, that is $\t(j)$, where
	$$j=\lb \begin{array}{cl}
		\inf\{k\ge i\,:\, \tilde\l(\t(k))= \tilde\l(\t(i)) - 1\} &\text{ if } \{k\ge i\,:\, \tilde\l(\t(k))= \tilde\l(\t(i)) - 1\} \neq \varnothing,\\
		\inf\{k\ge 1\,:\, \tilde\l(\t(k))= \tilde\l(\t(i)) - 1\} &\text{ otherwise}.
		\end{array}\rno$$
\end{itemize}

The quadrangulation $\q$ is then defined as the map whose set of vertices is $V(\t) \cup \{v_\bullet\}$, whose edges are the arcs we drew and whose root is the first arc drawn, oriented \textit{from} $\t(0)$ if $\eps=-1$ or \textit{toward} $\t(0)$ if $\eps=+1$ (see Figure \ref{cms}).

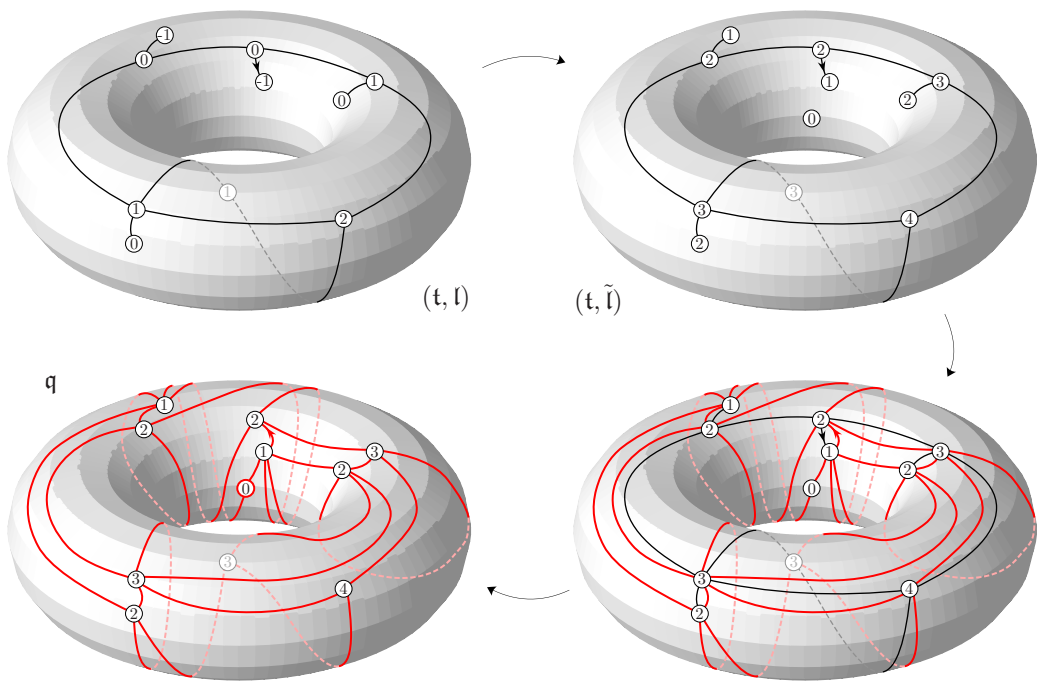
\tfig{cms}{The Chapuy-Marcus-Schaeffer bijection. In this example, $\eps = 1$. On the bottom-left picture, the vertex $v_\bullet$ has a thicker (red) borderline.}{
\psfrag{t}[][]{$(\t,\l)$}
\psfrag{u}[][]{$(\t,\tilde\l)$}
\psfrag{q}[][]{$\q$}
\psfrag{7}[][][.62]{-$1$}
\psfrag{8}[][][.62]{$0$}
\psfrag{9}[][][.62]{$1$}
\psfrag{0}[][][.62]{$2$}
\psfrag{1}[][][.62]{$3$}
\psfrag{2}[][][.62]{$4$}
\psfrag{d}[][][.62]{\textcolor{gris}{$1$}}
\psfrag{b}[][][.62]{\textcolor{gris}{$3$}}
}

Because of the way we drew the arcs of $\q$, we see that for any vertex $v\in V(\q)$, $\tilde\l(v)=d_\q(v_\bullet,v)$. When seen as a vertex in $V(\q)$, we write $\q(i)$ instead of $\t(i)$. In particular,
$$\{\q(i),0 \le i \le 2n\} = V(\q)\bs\{v_\bullet\}.$$

\bigskip

We end this section by giving an upper bound for the distance between two vertices $\q(i)$ and $\q(j)$, in terms of the labeling function $\l$:
\begin{equation}\label{lemmed}
d_\q(\q(i),\q(j)) \le \l(\t(i)) + \l(\t(j)) - 2 \max \lp \min_{ k \in \overrightarrow{\ent i j}} \l(\t(k)),\min_{ k \in \overrightarrow{\ent j i}} \l(\t(k)) \rp +2
\end{equation}
where we note, for $i\le j$, $\ent i j \de [i,j] \cap \Z = \{ i, i+1, \dots, j \}$, and
$$\overrightarrow{\ent i j} \de \lb \begin{array}{cll}
					\ent i j			&\text{ if } & i \le j,\\
					\ent i {2n} \cup \ent 0 j	&\text{ if } & j < i.
				    \end{array}\rno$$

We refer the reader to \cite[Lemma 4]{miermont09trm} for a detailed proof of this bound. The idea is the following: we consider the paths starting from $\t(i)$ and $\t(j)$ and made of the successive arcs going from vertices to their successors without crossing the $g$-tree. They are bound to meet at a vertex with label $m-1$, where
$$m \de \displaystyle\min_{ k \in \overrightarrow{\ent i j}} \l(\t(k)).$$
On Figure \ref{bound}, we see that the (red) plain path has length $\l(\t(i))-m+1$ and that the (purple) dashed one has length $\l(\t(j))-m+1$.

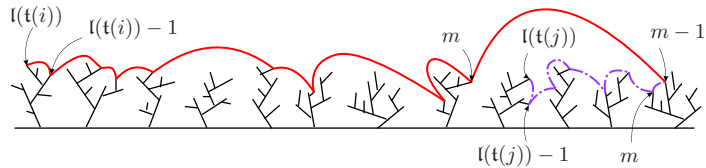
\tfig{bound}{Visual proof for \eqref{lemmed}. Both paths are made of arcs constructed as explained above.}{
\psfrag{i}[][][.8]{$\l(\t(i))$}
\psfrag{j}[][][.8]{$\l(\t(j))$}
\psfrag{u}[l][l][.8]{$\l(\t(i))-1$}
\psfrag{v}[][][.8]{$\l(\t(j))-1$}
\psfrag{m}[][][.8]{$m$}
\psfrag{1}[l][l][.8]{$m-1$}
}


\section{Decomposition of a $g$-tree}\label{decomp}

We investigate here more closely the structure of a $g$-tree $\t$. We call \textbf{scheme} a $g$-tree with no vertices of degree $1$ or $2$. Roughly speaking, a $g$-tree is a scheme in which every half-edge is replaced by a forest.

\subsection{Forests}\label{secfor}

\subsubsection{Formal definitions}

We adapt the standard formalism for plane trees---as found in \cite{neveu86apg} for instance---to forests. Let
$$\U \de \bigcup_{n=1}^{\infty} \N^n$$
where $\N \de \{1,2,\dots\}$. If $u \in \N^n$, we write $|u| \de n$. For $u=(u_1,\dots,u_n)$, $v=(v_1,\dots,v_p) \in \U$, we let $uv \de (u_1,\dots,u_n,v_1,\dots,v_p)$ be the concatenation of $u$ and $v$. If $w=uv$ for some $u,v\in\U$, we say that $u$ is a \bf{ancestor} of $w$ and that $w$ is a \bf{descendant} of $u$. In the case where $v \in \N$, we may also use the terms \bf{parent} and \bf{child} instead.

\begin{defi}
A \bf{forest} is a finite subset $\f \subset \U$ satisfying:
\begin{enumerate}[i.]
 \item there is an integer $t(\f) \ge 1$ such that $\f\cap \N = \ent 1 {t(\f)+1}$,
 \item if $u \in \f$, $|u| \ge 2$, then its parent belongs to $\f$,
 \item for every $u \in \f$, there is an integer $c_u(\f) \ge 0$ such that $ui \in \f$ \iff $1 \le i \le c_u(\f)$,
 \item $c_{t(\f)+1}(\f) =0$.
\end{enumerate}

The integer $t(\f)$ encountered in $i.$ and $iv.$ is called the \textbf{number of trees} of $\f$.
\end{defi}

We will see in a moment why we require $t(\f)+1$ to lie in $\f$. For $u=(u_1,\dots,u_p) \in \f$, we call $\a(u) \de u_1$ its oldest ancestor. A \textbf{tree} of~$\f$ is a level set for $\a$: for $1\le j \le t(\f)$, the $j$-th tree of~$\f$ is the set $\{u\in \f\,:\, \a(u)=j\}$. The integer $\a(u)$ hence
records which tree $u$ belongs to.
We call $\f\cap \N = \lb \a(u),\, u \in \f \rb$ the \bf{floor} of the forest $\f$.

\bigskip

For $u,v \in \f$, we write $u \sim v$ if either
\begin{itemize}
 \item $u$ is a parent or child of $v$, or
 \item $u,v \in \N$ and $|u-v| =1$.
\end{itemize}

It is convenient, when representing a forest, to draw edges between parents and their children, as well as between $i$ and $i+1$, for $i=1,2,\dots,t(\f)$, that is between $u$'s and $v$'s such that $u \sim v$ (see Figure \ref{wlforest}). We say that an edge drawn between a parent and its child is a \bf{tree edge} whereas an edge drawn between an $i$ and an $i+1$ will be called a \bf{floor edge}.

We call $\F_\sigma^m$ the set of all forests with $\sigma$ trees and $m$ tree edges, that is
$$\F_\sigma^m \de \lb \f\,:\, t(\f)=\sigma,\,|\f|= m+\sigma +1 \rb.$$

\begin{defi}
A \bf{well-labeled forest} is a pair $(\f,\l)$ where $\f$ is a forest and $\l:\f \to \Z$ is a function satisfying:
\begin{enumerate}[i.]
 \item for all $u\in \f \cap \N$, $\l(u) = 0$,
 \item if $u \sim v$, $|\l(u) - \l(v)| \le 1$.
\end{enumerate}
\end{defi}

Let
$$\fF_\sigma^m \de \lb (\f,\l):\, \f \in \F_\sigma^m \rb$$
be the set of well-labeled forests with $\sigma$ trees and $m$ tree edges.

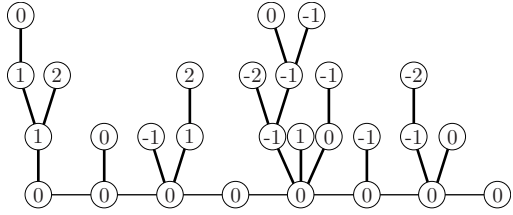
\tfig{wlforest}{An example of well-labeled forest from $\fF_7^{20}$.}{
\psfrag{1}[][][.8]{$1$}
\psfrag{2}[][][.8]{$2$}
\psfrag{3}[][][.8]{$3$}
\psfrag{4}[][][.8]{$4$}
\psfrag{0}[][][.8]{$0$}
\psfrag{9}[][][.8]{-$1$}
\psfrag{8}[][][.8]{-$2$}
}

\begin{rem}
For every forest in $\F_\sigma^m$, there are exactly $3^m$ admissible ways to label it: for all tree edges, one may choose any increment in $\{-1,0,1\}$. As a result, $|\fF_\sigma^m| = 3^m |\F_\sigma^m|$.
\end{rem}


\subsubsection{Encoding by contour and spatial contour functions}

There is a very convenient way to code forests and well-labeled forests. Let $\f \in \F_\sigma^m$ be a forest. Let us begin by defining its \textbf{facial sequence} $\f(0),\f(1),\dots,\f(2m+\sigma)$ as follows (see Figure \ref{dfs}): $\f(0) \de 1$, and for $0\le i \le 2m+\sigma-1$,
\begin{itemize}
 \item if $\f(i)$ has children that do not appear in the sequence $\f(0),\f(1),\dots,\f(i)$, then $\f(i+1)$ is the first of these children, that is $\f(i+1)\de \f(i)j_0$ where
 $$j_0 = \min\lb j\ge 1\,:\, \f(i)j \notin \lb \f(0),\f(1),\dots,\f(i) \rb \rb,$$
 \item otherwise, if $\f(i)$ has a parent (that is $|\f(i)| \ge 2$), then $\f(i+1)$ is this parent,
 \item if neither of these cases occur, which implies that $|\f(i)|=1$, then $\f(i+1) \de \f(i)+1$. 
\end{itemize}

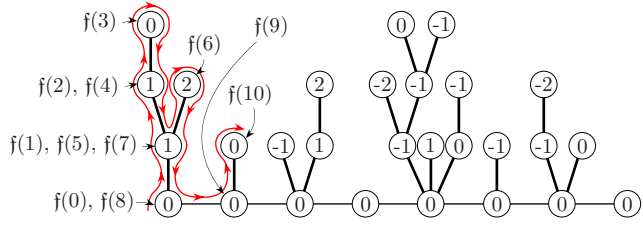
\tfig{dfs}{The facial sequence associated with the well-labeled forest from Figure \ref{wlforest}.}{
\psfrag{1}[][][.8]{$1$}
\psfrag{2}[][][.8]{$2$}
\psfrag{3}[][][.8]{$3$}
\psfrag{4}[][][.8]{$4$}
\psfrag{0}[][][.8]{$0$}
\psfrag{9}[][][.8]{-$1$}
\psfrag{8}[][][.8]{-$2$}
\psfrag{a}[r][r][.8]{$\f(0)$, $\f(8)$}
\psfrag{b}[r][r][.8]{$\f(1)$, $\f(5)$, $\f(7)$}
\psfrag{c}[r][r][.8]{$\f(2)$, $\f(4)$}
\psfrag{d}[r][r][.8]{$\f(3)$}
\psfrag{e}[][][.8]{$\f(6)$}
\psfrag{f}[][][.8]{$\f(9)$}
\psfrag{g}[][][.8]{$\f(10)$}
}

It is easy to see that each tree edge is visited exactly twice---once going from the parent to the child, once going the other way around---whereas each floor edge is visited only once---from some~$i$ to $i+1$. As a result, $\f(2m+\sigma) = t(\f) + 1$.

\bigskip

The \textbf{contour function} of $\f$ is the function $C_\f:[0,2m+\sigma]\to \R_+$ defined, for $0 \le i \le 2m+\sigma$, by
$$C_\f(i) \de |\f(i)| + t(\f) -\a\lp \f(i) \rp$$
and linearly interpolated between integer values (see Figure \ref{wlf}).

We can easily check that the function $C_\f$ entirely determines the forest $\f$. We see that $C_\f$ ranges in the set of paths of a simple random walk starting from $t(\f)$ and conditioned to hit~$0$ for the first time at $2m+\sigma$. This allows us to compute the cardinality of $\F_\sigma^m$:

\begin{lem}\label{forests}
Let $\sigma \ge 1$ and $m \ge 0$ be two integers. The number of forests with $\sigma$ trees and $m$ tree edges is:
$$\big| \F_\sigma^m \big| = \frac \sigma {2m+\sigma} \ 2^{2m+\sigma} \ \P(S_{2m+\sigma}=\sigma) = \frac \sigma {2m+\sigma} \choo {2m+\sigma} m,$$
where $(S_i)_{i \ge 0}$ is a simple random walk on $\Z$.
\end{lem}

\begin{pre}
Shifting the contour functions, we see that $\lm \F_\sigma^m \rm$ is the number of different paths of a simple random walk starting from $0$ and conditioned to hit $-\sigma$ for the first time at $2m+\sigma$. We have
\begin{align*}
\big| \F_\sigma^m \big|  &= 2^{2m+\sigma} \, \P \lp S_{2m+\sigma}=-\sigma \,;\, \forall i \in \ent 0 {2m+\sigma-1}, \, S_i > -\sigma \rp\\
		&= \frac \sigma {2m+\sigma} \ 2^{2m+\sigma} \ \P(S_{2m+\sigma}=\sigma),
\end{align*}
where the second equality is an application of the so-called cycle lemma (see for example \cite[Lemma 2]{bertoin03ptf}). The second equality of the lemma is obtained by seeing that $S_{2m+\sigma}=\sigma$ if and only if the walk goes exactly $m+\sigma$ times up and $m$ times down.
\end{pre}

Now, if we have a well-labeled forest $(\f,\l)$, the contour function $C_\f$ enables us to recover~$\f$. To record the labels, we use the \textbf{spatial contour function} $L_{\f,\l}:[0,2m+\sigma]\to \R$ defined, for $0 \le i \le 2m+\sigma$, by
$$L_{\f,\l}(i) \de \l(\f(i))$$
and linearly interpolated between integer values (see Figure \ref{wlf}). The \textbf{contour pair} $(C_\f,L_{\f,\l})$ then entirely determines $(\f,\l)$.

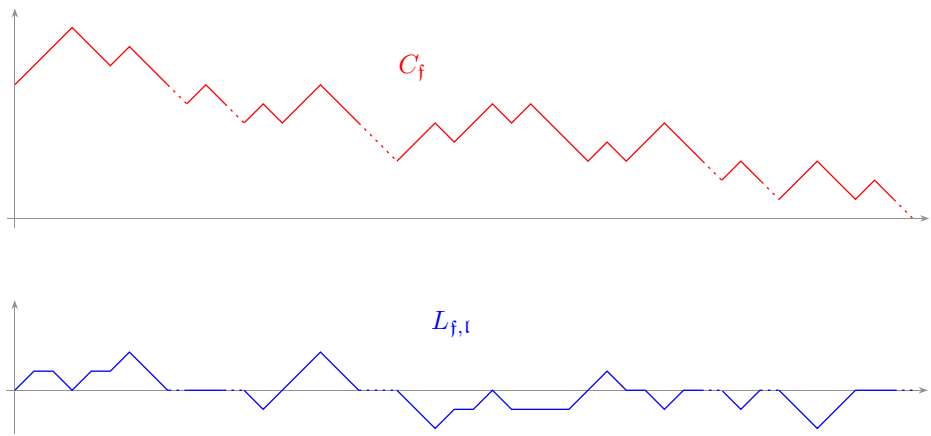
\tfig{wlf}{The contour pair of the well-labeled forest appearing in Figures \ref{wlforest} and \ref{dfs}. The paths are dashed on the intervals corresponding to floor edges.}{
\psfrag{C}[][]{\textcolor{red}{$C_{\f}$}}
\psfrag{L}[][]{\textcolor{blue}{$L_{\f,\l}$}}
}


\subsection{Scheme}

\subsubsection{Extraction of the scheme out of a $g$-tree}

\begin{defi}
We call \bf{scheme} of genus $g$ a $g$-tree with no vertices of degree one or two. A scheme is said to be \bf{dominant} when it only has vertices of degree exactly three.
\end{defi}

\begin{rem}
The Euler characteristic formula easily shows that the number of schemes of genus~$g$ is finite. We call $\Sg$ the set of all schemes of genus $g$ and $\Sg^*$ the set of all dominant schemes of genus~$g$.
\end{rem}

It was explained in \cite{chapuy07brm} how to extract the scheme out of a $g$-tree $\t$. Let us recall now this operation. By iteratively deleting all its vertices of degree $1$, we are left with a---non-necessarily rooted---$g$-tree. If the root has been removed, we root this new $g$-tree on the first remaining half-edge following the actual root in the facial order of $\t$.

The vertices of degree $2$ in the new $g$-tree are organized into maximal chains connected together at vertices of degree at least $3$. We replace each of these maximal chains by a single new edge. The edge replacing the chain containing the root is chosen to be the final root (with the same orientation).

By construction, the map $\s$ we obtain is a scheme of genus $g$, which we call the scheme of the $g$-tree~$\t$. The vertices of $\t$ that remain vertices in the scheme $\s$ are called the \bf{nodes} of~$\t$. See Figure~\ref{scheme}.

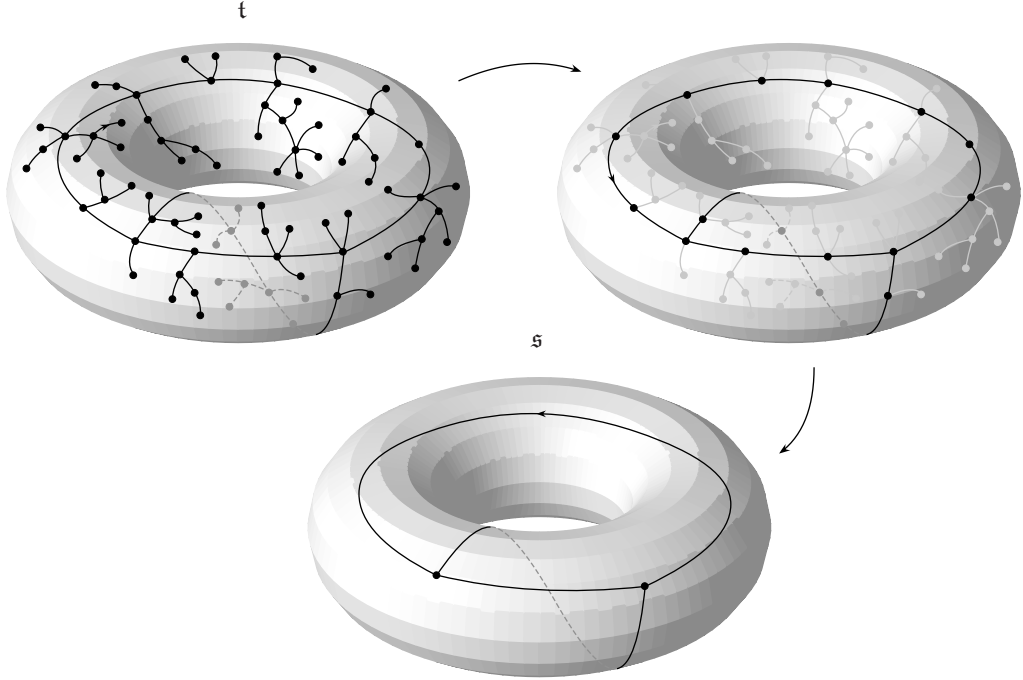
\tfig{scheme}{Extraction of the scheme $\s$ out of the $g$-tree $\t$.}{
\psfrag{t}[][]{$\t$}
\psfrag{s}[][]{$\s$}
}


\subsubsection{Decomposition of a $g$-tree}

When iteratively removing vertices of degree $1$, we actually remove whole trees. Let $\ccc_1$, $\ccc_2$, \dots, $\ccc_k$ be one of the maximal chains of half-edges linking two nodes. The trees that we remove, appearing on the left side of this chain, connected to one of the $\ccc_i^-$'s, form a forest---with~$k$ trees---as defined in Section \ref{secfor}. Beware that the tree connected to $\ccc_k^+$ is not a part of this forest; it will be the first tree of some other forest. Remember that the forests we consider always end by a single vertex not considered to be a tree. This chain being later replaced by a single half-edge of the scheme, we see that a $g$-tree $\t$ can be decomposed into its scheme $\s$ and a collection of forests $(\f^\e)_{\e\in \vec E(\s)}$. Recall that $\vec E(\s)$ is the set of all half-edges of $\s$. 

For $\e\in \vec E(\s)$, let us define the integers $m^\e \ge 0$ and $\sigma^\e \ge 1$ by
\begin{equation}\label{msig}
\f^{\e} \in \F_{\sigma^\e}^{m^\e},
\end{equation}
so that $m^\e$ records the ``size'' of the forest attached on the half-edge $\e$ and $\sigma^\e$ its ``length.''

In order to recover $\t$ from $\s$ and these forests, we need to record the position its root. It may be seen as a half-edge of the forest $\f^{\e_*}$ corresponding to the root $\e_*$ of $\s$. We code it by the integer
\begin{equation}\label{conu}
u \in \llbracket 0, \, 2m^{\e_*} + \sigma^{\e_*} \llbracket
\end{equation}
for which this half-edge links $\f^{\e_*}(u)$ to $\f^{\e_*}(u+1)$.

For every half-edge $\e\in \vec E(\s)$, if we call $\bar\e$ its reverse, we readily obtain the relation:
\begin{equation}\label{consig}
\sigma^{\bar\e}=\sigma^{\e}.
\end{equation}

\bigskip

This decomposition may be inverted. Let us suppose that we have a scheme $\s$ and a collection of forests $(\f^\e)_{\e\in \vec E(\s)}$. Let us define the integers $m^\e$'s and $\sigma^\e$'s by \eqref{msig} and suppose they satisfy \eqref{consig}. Let again $0 \le u < 2m^{\e_*} + \sigma^{\e_*}$ be an integer. Then we may construct a $g$-tree as follows. First, we replace every edge $\{\e,\bar\e\}$ by a chain of $\sigma^\e = \sigma^{\bar\e}$ edges. Then, for every half-edge $\e\in \vec E(\s)$, we replace the chain of half-edges corresponding to it by the forest $\f^\e$, in such a way that its floor\footnote{The floor of a forest $\f$ is naturally oriented from $1$ to $t(\f)+1$. The forest $\f^\e$ is then grafted ``to the left side'' of $\e$.} matches with the chain. Finally, we find the root inside $\f^{\e_*}$ thanks to the integer~$u$.

This discussion is summed up by the following proposition. The factor $1/2$ in the last statement comes from the fact that the floor of $\f^\e$ and that of $\f^{\bar\e}$ are overlapping in the $g$-tree, thus their edges should be counted only once. 

\begin{prop}\label{decompnolab}
The above construction provides us with a bijection between the set of all $g$-trees and the set of all triples $\lp \s,(\f^\e)_{\e\in \vec E(\s)},u\rp$ where $\s \in \Sg$ is a scheme (of genus $g$), the forests $\f^\e\in \F_{\sigma^\e}^{m^\e}$ satisfy \eqref{consig} and $u$ satisfies \eqref{conu}.

Moreover, $g$-trees with $n$ edges correspond to triples satisfying $\sum_{\e\in \vec E(\s)} \lp m^\e + \frac 1 2 \sigma^\e \rp = n$.
\end{prop}


\subsubsection{Decomposition of a well-labeled $g$-tree}

We now deal with a well-labeled $g$-tree. We will need the following definitions:

\begin{defi}
We call \textbf{Motzkin path} a sequence of the form $(M_n)_{0 \le n \le \sigma}$ for some $\sigma \ge 0$ such that $M_0=0$ and for $0 \le i \le \sigma -1$, $M_{i+1} - M_i \in \{-1,0,1\}$. We write $\sigma(M) \de \sigma$ its lifetime, and $\hat M \de M_{\sigma(M)}$ its final value.
\end{defi}

A \textbf{Motzkin bridge} of lifetime $\sigma$ from $l_1 \in \Z$ to $l_2 \in \Z$ is an element of the set
$$\M_{[0,\sigma]}^{l_1 \to l_2} \de \lb l_1 + M\,:\, M \text{ Motzkin path such that }\sigma(M)=\sigma,\, \hat M=l_2-l_1 \rb.$$
We say that $(M_n)_{n\ge 0}$ is a \bf{simple Motzkin walk} if it is defined as the sum of i.i.d. random variables with law $\frac 13 (\delta_{-1} + \delta_{0} + \delta_{1})$.

\begin{rem}
We then have 
$$\big| \M_{[0,\sigma]}^{l_1 \to l_2} \big| = 3^\sigma \, \P(M_\sigma = l_2 - l_1 )$$
where $(M_i)_{i \ge 0}$ is a simple Motzkin walk.
\end{rem}

When decomposing a well-labeled $g$-tree $(\t,\l)$ into a triple $( \s,(\f^\e),u)$ according to Proposition~\ref{decompnolab}, every forest $\f^\e$ naturally inherits a labeling function noted $\tilde \l^\e$ from $\l$. In general, the forest $(\f^\e,\tilde\l^\e)$ is not well-labeled, because the labels of its floor have no reason to be equal to $0$. We will transform it into a Motzkin bridge $\mM^\e$ starting from $0$ and a well-labeled forest $(\f^\e,\l^\e)$. The Motzkin bridge records the floor labels shifted in order to start from $0$: for $0\le i \le t(\f^\e)$, $\mM^\e(i) \de \tilde\l^\e(i+1)-\tilde\l^\e(1)$, where, on the right-hand side, we used the notation $\{1, 2, \dots, t(\f^\e)+1\}$ for the floor of $\f^\e$. The well-labeled forest is obtained by shifting all the labels tree by tree in such a way that the root label of any tree is $0$: for all $w\in\f^\e$, $\l^\e(w)\de \tilde\l^\e(w)-\tilde\l^\e(\a(w))$.

We thus decompose the well-labeled $g$-tree $(\t,\l)$ into its scheme $\s$, a collection $(\mM^\e)_{\e\in\vec E(\s)}$ of Motzkin bridges started at $0$, a collection $(\f^{\e},\l^\e)_{\e \in \vec E(\s)}$ of well-labeled forests and an integer~$u$, as shown on Figure \ref{structure}.

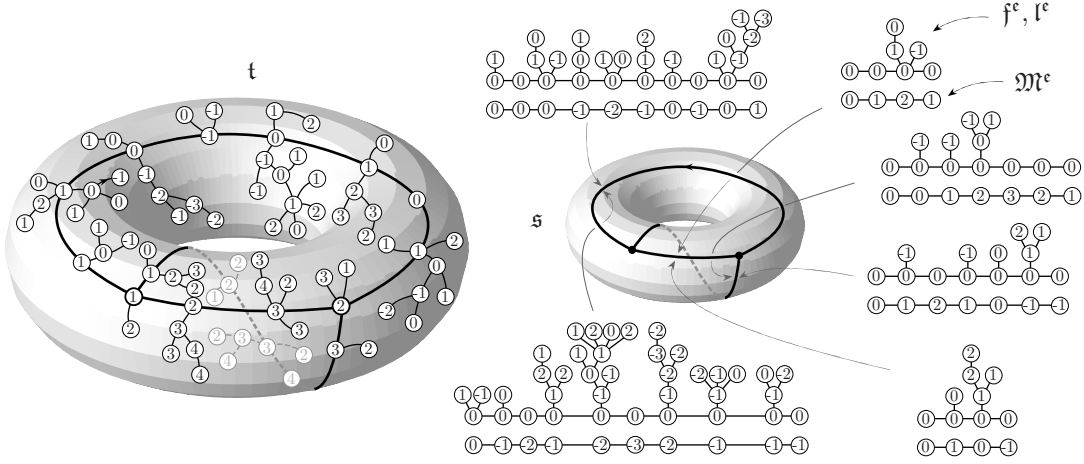
\tfig{structure}{Decomposition of a well-labeled $g$-tree $\t$ into its scheme $\s$, the collection of Motzkin bridges $(\mM^\e)_{\e\in\vec E(\s)}$,  and the collection of well-labeled forests $(\f^{\e},\l^\e)_{\e \in \vec E(\s)}$. In this example, the integer $u=50$. The two nodes of $\t$ are more thickly outlined.}{
\psfrag{6}[][][.62]{-$4$}
\psfrag{7}[][][.62]{-$3$}
\psfrag{8}[][][.62]{-$2$}
\psfrag{9}[][][.62]{-$1$}
\psfrag{0}[][][.62]{$0$}
\psfrag{1}[][][.62]{$1$}
\psfrag{2}[][][.62]{$2$}
\psfrag{3}[][][.62]{$3$}
\psfrag{4}[][][.62]{$4$}
\psfrag{a}[][][.62]{\textcolor{gris}{$1$}}
\psfrag{b}[][][.62]{\textcolor{gris}{$2$}}
\psfrag{c}[][][.62]{\textcolor{gris}{$3$}}
\psfrag{d}[][][.62]{\textcolor{gris}{$4$}}
\psfrag{t}[][]{$\t$}
\psfrag{s}[][]{$\s$}
\psfrag{f}[r][r]{$\f^{\e},\l^\e$}
\psfrag{M}[r][r]{$\mM^{\e}$}
}

For $\e\in \vec E(\s)$, we define the integer $l^\e\in \Z$ to be such that
\begin{equation}\label{msigl}
\mM^\e \in \M_{[0,\sigma^\e]}^{0\to l^\e}.
\end{equation}
It records the spatial displacement made along the half-edge $\e$. Because the floor of $\f^\e$ overlaps the floor of $\f^{\bar \e}$ in the $g$-tree, $\mM^\e$ and $\mM^{\bar\e}$ read the same labels in opposite direction:
\begin{equation}\label{mm}
\mM^{\bar\e}(i) = \mM^\e(\sigma^\e-i) - l^\e.
\end{equation}

In particular, $l^{\bar\e}= -l^\e$. But this is not the only constraints on the family $(l^\e)_{\e\in \vec E(\s)}$. These will be easier to understand while looking at vertices instead of edges. For every vertex $v \in V(\s)$, we let $l^v$ be the label of the corresponding node shifted in such a way that $l^{\e_*^-}=0$. We have the following relation between $(l^\e)_{\e\in \vec E(\s)}$ and $(l^v)_{v\in V(\s)}$: for all $\e\in \vec E(\s)$,
\begin{equation}\label{ev}
l^\e = l^{\e^+}-l^{\e^-},
\end{equation}
so that the family $(l^v)_{v\in V(\s)}$ entirely determines $(l^\e)_{\e\in \vec E(\s)}$. Because of the choice we made, $l^{\e_*^-}=0$, it is easy to see that $(l^\e)_{\e\in \vec E(\s)}$ determines $(l^v)_{v\in V(\s)}$ as well.

It now becomes clear that the only constraint on $(l^\e)_{\e\in \vec E(\s)}$ is to be obtained from a family $(l^v)_{v\in V(\s)}$ by the relations \eqref{ev}.

\bigskip

Let $\s$ be a scheme, $(\mM^\e)_{\e\in \vec E(\s)}$ be a family of Motzkin bridges started from $0$, $(\f^{\e},\l^\e)_{\e \in \vec E(\s)}$ be a family of well-labeled forests, and $u$ be an integer. Let the integers $m^\e$'s, $\sigma^\e$'s and $l^\e$'s be defined by \eqref{msig} and \eqref{msigl}. We will say that the quadruple $\big( \s, (\mM^\e)_{\e\in \vec E(\s)}, (\f^{\e},\l^\e)_{\e \in \vec E(\s)},u \big)$ is \textit{compatible} if the integers $\sigma^\e$'s satisfy the constraints \eqref{consig}, the Motzkin bridges $\mM^\e$'s satisfy \eqref{mm}, the integers $l^\e$'s can be obtained from a family $(l^v)_{v\in V(\s)}$ by the relations \eqref{ev}, and $u$ satisfies \eqref{conu}.

Let suppose now that we have a compatible quadruple $\big( \s, (\mM^\e)_{\e\in \vec E(\s)}, (\f^{\e},\l^\e)_{\e \in \vec E(\s)},u \big)$. We may reconstruct a well-labeled $g$-tree as follows. We begin by suitably relabeling the forests. For every half-edge $\e$, first, we shift the labels of $\mM^\e$ by $l^{\e^-}$ so that it becomes a bridge from $l^{\e^-}$ to $l^{\e^+}$. Then, we shift all the labels of $(\f^\e,\l^\e)$ tree by tree according to the Motzkin bridge: precisely, we change $\l^\e$ into $w\in\f^\e \mapsto l^{\e^-} + \mM^\e(\a(w)-1) + \l^\e(w)$. Then, we replace the half-edge~$\e$ by this forest, as in the previous section. As before, we find the root thanks to $u$. Finally, we shift all the labels for the root label to be equal to $0$. This discussion is summed up by the following proposition.

\begin{prop}
The above construction provides us with a bijection between the set of all well-labeled $g$-trees and the set of all compatible quadruples $\big( \s,(\mM^\e)_{\e\in \vec E(\s)}, (\f^{\e},\l^\e)_{\e \in \vec E(\s)},u \big)$.

Moreover, $g$-trees with $n$ edges correspond to quadruples satisfying $\sum_{\e\in \vec E(\s)} \lp m^\e + \frac 1 2 \sigma^\e \rp = n$.
\end{prop}

If we call $(C^\e,L^\e)$ the contour pair of $(\f^\e,\l^\e)$, then we may retrieve the oldest ancestor of $\f^\e(i)$ thanks to $C^\e$ by the relation
$$\a\big(\f^\e(i)\big) -1 = \sigma^\e - \underline C^\e(i),$$
where we use the notation
$$\underline{X}_s \de \inf_{[0, s]} X$$
for any process $(X_s)_{s\ge 0}$. The function
\begin{equation}\label{lele}
\L^\e \de \Big( L^\e(t) + \mM^\e\big( \sigma^\e - \underline C^\e(t) \big) \Big)_{0\le t \le 2m^\e +\sigma^\e}
\end{equation}
then records the labels of the forest $\f^\e$, once shifted tree by tree according to the Motzkin bridge~$\mM^\e$. This function will play an important part in Section \ref{secpro}.

\bigskip

Through the Chapuy-Marcus-Schaeffer bijection, a uniform random quadrangulation corresponds to a uniform random well-labeled $g$-tree. In order to investigate the scaling limit of the latter, we will proceed in two steps. First, we consider the scaling limit of its \textit{structure}, consisting in its scheme along with the integers $m^\e$'s, $\sigma^\e$'s, $l^v$'s and $u$ previously defined. Then, we deal with its Motzkin bridges and forests conditionally given its structure.


\section{Convergence of the structure of a uniform well-labeled $g$-tree}\label{seccv}
\subsection{Preliminaries}

We investigate here the convergence of the integers previously defined, suitably rescaled, in the case of a uniform random well-labeled $g$-tree with $n$ vertices. Let $(\t_n,\l_n)$ be uniformly distributed over the set $\T_n$ of well-labeled $g$-trees with $n$ vertices. We call its scheme $\s_n$ and we define
$$(\mM_n^{\e})_{\e \in \vec E(\s_n)},\
(\f_n^{\e},\l_n^\e)_{\e \in \vec E(\s_n)},\ 
(m_n^\e)_{\e\in \vec E(\s_n)},\ 
(\sigma_n^\e)_{\e\in \vec E(\s_n)},\ 
(l_n^\e)_{\e\in \vec E(\s_n)},\ 
(l_n^v)_{v\in V(\s_n)},\ \text{and}\ 
u_n$$
as in the previous section. We know that the right scalings are $2n$ for sizes, $\sqrt{2n}$ for distances in the $g$-tree, and $\g$ for spatial displacements\footnote{Recall that $\gamma \de \lp\frac {8}9 \rp^{\frac 14}$.}, so we set:
$$m_{(n)}^\e \de \frac{2m_n^\e + \sigma^\e_n}{2n}\scomp
\sigma_{(n)}^\e \de \frac{\sigma_n^\e}{\sqrt{2n}}\scomp
l_{(n)}^\e \de \frac{l_n^\e}{\g}\scomp
l_{(n)}^v \de \frac{l_n^v}{\g}\sand
u_{(n)} \de \frac{u_n}{2n}.$$

\begin{rem}
Throughout this paper, the notations with a parenthesized $n$ will always refer to suitably rescaled objects---as in the definitions above.
\end{rem}

As sensed in the previous section, it will be more convenient to work with $l^v$'s instead of $l^\e$'s. We use the notation $\Zp\de\{0,1,\dots\}$ for the set of non-negative integers. For any scheme $\s \in \Sg$, we define the set $\CC_n(\s)$ of quadruples $(m,\sigma,l,u)$ lying in $\Zp^{\vec E(\s)}\times\N^{\vec E(\s)}\times\Z^{V(\s)}\times\Zp$
such that:
\begin{itemize}
	\item $\forall \e \in \vec E(\s),\, \sigma^{\bar\e}=\sigma^{\e}$,
	\item $l^{\e_*^-}=0$,
	\item $0\le u \le 2m^{\e_*} + \sigma^{\e_*} -1$,
	\item $\sum_{\e\in \vec E(\s)} \lp m^\e + \frac 12 \sigma^\e\rp =n$.
\end{itemize}

This is the set of integers satisfying the constraints discussed in the previous section for a well-labeled $g$-tree with $n$ edges. For $(m,\sigma,l,u)\in \CC_n(\s)$, we will compute the probability that $\s_n=\s$ and $(m_n,\sigma_n,l_n,u_n)=(m,\sigma,l,u)$. A $g$-tree has such features if and only if its scheme is~$\s$ and, for every $\e\in\vec E(\s)$, the path $\mM^\e$ is a Motzkin bridge from $0$ to $l^\e=l^{\e^+}\!\!-\,l^{\e^-}$ on $[0,\sigma^\e]$, and the well-labeled forest $(\f^\e,\l^\e)$ lies in $\fF_{\sigma^\e}^{m^\e}$.

Moreover, because of the relation \eqref{mm}, the Motzkin bridges $(\mM^\e)_{\e\in\vec E(\s)}$ are entirely determined by $(\mM^\e)_{\e\in\ori E(\s)}$, where $\ori E(\s)$ is any orientation of $\vec E(\s)$.
Using Lemma \ref{forests}, we obtain
\begin{align}
\P\lp \s_n=\s,\,(m_n,\sigma_n,l_n,u_n)=(m,\sigma,l,u) \rp\notag\\
	&\hspace{-2cm}= \frac 1 {\lm \T_n \rm} \prod_{\e\in\ori{E}(\s)} \big| \M_{[0,\sigma^\e]}^{0\to l^\e} \big| \big| \fF_{\sigma^\e}^{m^\e} \big| \big|\fF_{\sigma^{\bar\e}}^{m^{\bar\e}}\big| \notag\\
	&\hspace{-2cm}= \frac {12^n} {\lm \T_n \rm} \prod_{\e\in{\vec E}(\s)} \frac {\sigma^\e} {2m^\e+\sigma^\e} \, \P(S_{2m^\e+\sigma^\e}=\sigma^\e) \prod_{\e\in\ori{E}(\s)} \P(M_{\sigma^\e}=l^\e)\label{proba}
\end{align}
where $(S_i)_{i\ge 0}$ is a simple random walk on $\Z$ and $(M_i)_{i \ge 0}$ is a simple Motzkin walk.

\bigskip

We will need the following local limit theorem (see \cite[Theorems VII.1.6 and VII.3.16]{petrov75sir}) to estimate the probabilities above. We call $\ga$ the density of a standard Gaussian random variable: $\ga(x) = \frac 1 {\sqrt{2\pi}} \,e^{-\frac {x^2} 2}$.

\begin{prop}\label{petrov}
Let $(X_i)_{i \ge 0}$ be a sequence of i.i.d. integer-valued centered random variables with a moment of order $r_0$ for some $r_0\geq 3$. Let $\eta^2 \de \var(X_1)$, $h$ be the maximal span\footnote{\label{maxspan}We call \textbf{maximal span} of an integer-valued random variable $X$ the greatest $h \in \N$ for which there exists an integer $a$ such that a.s. $X \in a + h\Z$.} of $X_1$ and the integer $a$ be such that a.s. $X_1 \in a + h\Z$. We define $\Sigma_k \de \sum_{i=0}^k X_i$, and we write $Q^\Sigma_k(i) \de \P(\Sigma_k = i)$.

\begin{enumerate}
	\item We have
	   $$\sup_{i \in ka + h\Z} \lm \frac \eta h \, \sqrt{k} \, Q^\Sigma_k(i) - \ga \lp \frac {i} {\eta \sqrt{k}} \rp \rm = o \lp k^{-1/2} \rp.$$
	\item For all $2 \le r\le r_0$, there exists a constant $C$ such that for all $i \in \Z$ and $k \ge 1$,
	   $$\lm \frac \eta h \, \sqrt{k} \, Q^\Sigma_k(i) \rm \le \frac C {1+\lm \frac {i} {\eta \sqrt{k}} \rm^r}.$$
\end{enumerate}
\end{prop}

\begin{pre}
The first part of this theorem is merely \cite[Theorem VII.1.6]{petrov75sir} applied to the variables $\frac 1 h (X_k - a)$, which have $1$ as maximal span. The second part is an easy consequence of \cite[Theorem VII.3.16]{petrov75sir}.
%
\end{pre}

In what follows, we will always use the notation $S$ for simple random walks, $M$ for simple Motzkin walks, and $\Sigma$ for any other random walks. We will use this theorem with $S$ and $M$: we find $(\eta,h) = (1,2)$ for $S$ and $(\eta,h) = (\sqrt{ 2 / 3},1)$ for $M$. In both cases, we may take $r$ as large as we want.


\subsection{Result}

Recall that $\Sg^*$ is the set of all dominant schemes of genus $g$, that is schemes with only vertices of degree $3$. We call $\ga_a$ the density of a centered Gaussian variable with variance $a$, as well as $\ga'_a$ its derivative:
$$\ga_a(x) \de \frac 1 {\sqrt{a}} \, \ga \lp \frac x {\sqrt a} \rp \sand \ga'_a(x) = - \frac x {a^{3/2}}\, \ga \lp \frac x {\sqrt a}\rp.$$

For any $\s\in \Sg$, we identify an element $(m,\sigma,l,u) \in \R_+^{\vec E(\s)\bs\{\e_*\}}\times (\R_+^*)^{\ori E(\s)}\times \R^{V(\s)\bs\{\e^-_*\}}\times \R_+$ with an element of $\R_+^{\vec E(\s)}\times (\R_+^*)^{\vec E(\s)}\times \R^{V(\s)}\times \R_+$ by setting:
\addtocounter{equation}{1}
\newcounter{eqit}
\setcounter{eqit}{\theequation}
\begin{itemize}
	\item $m^{\e_*} \de 1 - \sum_{\e\in \vec E(\s)\bs\{\e_*\}} m^\e$, \hspace*{\fill}(\theeqit.1)
	\item for every $\e \in \ori E(\s),\, \sigma^{\bar\e}\de \sigma^{\e}$, \hspace*{\fill}(\theeqit.2)
	\item $l^{\e_*^-}\de0$. \hspace*{\fill}(\theeqit.3)
\end{itemize}

We write
$$\Delta_\s \de \lb (x_\e)_{\e\in\vec E(\s)} \in [0,1]^{\vec E(\s)},\, \textstyle\sum_{\e\in\vec E(\s)} x_\e =1 \rb$$
the simplex of dimension $|\vec E(\s)|-1$. Note that the vector $m$ lies in the simplex $\Delta_\s$ as long as $m^{\e_*} \ge 0$. We define the probability $\mu$ by, for all non-negative measurable function $\varphi$ on $\bigcup_{\s\in\Sg} \{\s\}\times\Delta_\s\times (\R_+^*)^{\vec E(\s)}\times \R^{V(\s)}\times [0,1]$,
$$\mu(\varphi) = \frac1\Upsilon \sum_{\s\in\Sg^*} \int_{\mathcal{S}^\s} d\mathcal{L}^\s \ 
		\1{\lb m^{\e_*}\ge 0,\,u< m^{\e_*}\rb}\ \varphi\lp \s, m,\sigma, l,u \rp
		\prod_{\e\in{\vec E}(\s)} - \ga'_{m^\e} \lp \sigma^\e \rp
		\prod_{\e\in\ori{E}(\s)} \ga_{\sigma^\e} \lp l^\e \rp,$$
where $d\mathcal{L}^\s= d(m^\e)\, d(\sigma^\e)\, d(l^v)\, du$ is the Lebesgue measure on the set
$$\mathcal{S}^\s \de [0,1]^{\vec E(\s)\bs\{\e_*\}}\times(\R_+^*)^{\ori E(\s)}\times \R^{V(\s)\bs\{\e^-_*\}}\times[0,1]$$
and
\begin{equation}\label{upsi}
\Upsilon = \sum_{\s\in\Sg^*} \int_{\mathcal{S}^\s} d\mathcal{L}^\s \ \1{\lb m^{\e_*}\ge 0,\,u< m^{\e_*}\rb}\ 
		\prod_{\e\in{\vec E}(\s)} - \ga'_{m^\e} \lp \sigma^\e \rp
		\prod_{\e\in\ori{E}(\s)} \ga_{\sigma^\e} \lp l^\e \rp
\end{equation}
is a normalization constant. We may now state the main result of this section.

\begin{prop}\label{cvint}
The law of the random variable
$$\lp \s_n, \big( m_{(n)}^\e\big)_{\e\in \vec E(\s_n)}, \big(\sigma_{(n)}^\e\big)_{\e\in \vec E(\s_n)}, \big( l_{(n)}^v\big)_{v\in V(\s_n)}, u_{(n)} \rp$$
converges weakly toward the probability $\mu$.
\end{prop}

\begin{pre}
Let $\varphi$ be a bounded continuous function on the set
$$\bigcup_{\s\in\Sg} \{\s\}\times\Delta_\s\times (\R_+^*)^{\vec E(\s)}\times \R^{V(\s)}\times [0,1].$$
We need to look at the convergence of
$$\mathbb{E}_n \de \E{\varphi\lp \s_n, \big( m_{(n)}^\e\big)_{\e\in \vec E(\s_n)}, \big(\sigma_{(n)}^\e\big)_{\e\in \vec E(\s_n)}, \big( l_{(n)}^v\big)_{v\in V(\s_n)}, u_{(n)} \rp}.$$

1) Let $n\in \N$. For the time being, we identify $(m,\sigma,l,u) \in \Zp^{\vec E(\s)\bs\{\e_*\}}\times \N^{\ori E(\s)}\times \Z^{V(\s)\bs\{\e^-_*\}}\times \Zp$ with an element of $\Zp^{\vec E(\s)}\times \N^{\vec E(\s)}\times \Z^{V(\s)}\times \Zp$ by (\theeqit.2), (\theeqit.3), and
\begin{itemize}
	\item $m^{\e_*}(n) \de n - \sum_{\e\in \vec E(\s)\bs\{\e_*\}} m^\e - \sum_{\e\in \ori E(\s)}\sigma^\e$, \hspace*{\fill}(\theeqit.1')
\end{itemize}
instead of (\theeqit.1), which may be seen as its discrete counterpart. This is an element of $\CC_n(\s)$ provided that $m^{\e_*}(n)\ge 0$ and $0\le u < 2m^{\e_*}(n) + \sigma^{\e_*}$. Beware that here the definition of $m^{\e_*}(n)$ actually depends on $n$. It also depends on $\sigma$ but we chose not to let it figure in the notation for space reasons.

For any vector $x =(x_1,x_2,\dots, x_k) \in \R^k$, we note $\lf x \rf$ the vector $\lp \lf x_1 \rf, \lf x_2 \rf,\dots, \lf x_k \rf\rp \in \Z^k$. Note that for $m\in \R_+^{\vec E(\s)\bs\{\e_*\}}$, $\lf m \rf^{\e_*}\!(n)$ is well defined through (\theeqit.1'). Until further notice, we will write $\lf m\rf^{\e_*}$ for $\lf m\rf^{\e_*}\!(n)$, which $n$ being implicit. So when we write $\lf m\rf$, we mean the vector such that $\lf m\rf^\e = \lf m^\e\rf$ for $\e \neq \e_*$ and $\lf m\rf^{\e_*} = \lf m\rf^{\e_*}\!(n)$. 
Using \eqref{proba}, we find
\begin{align*}
\mathbb{E}_n &= \frac{12^n}{\lm \T_n \rm}\, \sum_{\s\in\Sg}\ \sum_{(m,\sigma,l,u) \in \CC_n(\s)}
		\varphi\lp \s,\frac {2m+\sigma} {2n},\frac \sigma {\sqrt{2n}},\frac l \g,\frac u {2n}\rp
		\prod_{\e\in{\vec E}(\s)}\! h_1(m^\e,\sigma^\e)
		\prod_{\e\in\ori{E}(\s)}\! h_2(\sigma^\e,l^\e),
\end{align*}
where
$$h_1(m^\e,\sigma^\e) \de \frac {\sigma^\e} {2m^\e+\sigma^\e} \, Q^S_{2m^\e+\sigma^\e}(\sigma^\e) \, \1{\lb \sigma^\e\ge 1 \rb} \sand
  h_2(\sigma^\e,l^\e) \de Q^M_{\sigma^\e}(l^\e) \, \1{\lb \sigma^\e\ge 1 \rb}.$$
Writing the sum over $\CC_n(\s)$ in the form of an integral, we obtain
$$\mathbb{E}_n	= \frac{12^n}{\lm \T_n \rm} \sum_{\s\in\Sg} \int_{\tilde{\mathcal{S}}^\s} \! d\tilde{\mathcal{L}}^\s \ \1{\mathcal{E}_n^\s}(m,\sigma,u)\ \varphi_{\lf  \cdot \rf}
		\prod_{\e\in{\vec E}(\s)} h_1({\lf m \rf}^\e, \lf \sigma \rf^\e) \prod_{\e\in\ori{E}(\s)} h_2\big(\lf \sigma \rf^\e,\lf l \rf^{\e^+}\!\!- {\lf l \rf}^{\e^-}\big),$$
where $\varphi_{\lf \cdot \rf}$ stands for
$$\varphi\bigg( \s,\frac {2\lf m \rf + \lf \sigma \rf} {2n},\frac {\lf \sigma \rf} {\sqrt{2n}},\frac {\lf l \rf} \g, \frac {\lf u \rf} {2n}\bigg),$$
$d\tilde{\mathcal{L}}^\s$ is the Lebesgue measure on the set
$\tilde{\mathcal{S}}^\s \de \R_+^{\vec E(\s)\bs\{\e_*\}}\times(\R_+^*)^{\ori E(\s)}\times \R^{V(\s)\bs\{\e^-_*\}}\times\R_+$
and
$$\mathcal{E}_n^\s \de \lb (m,\sigma,u)\in \R_+^{\vec E(\s)\bs\{\e_*\}}\times(\R_+^*)^{\ori E(\s)}\times\R_+ \,:\, 
				\lf m\rf^{\e_*}\!(n)\ge 0,\, u < 2\lf m \rf^{\e_*}\!(n) + \lf \sigma \rf^{\e_*}\rb.$$
Finally, the changes of variables $m \mapsto nm$, $\sigma \mapsto \sqrt{2n}\,\sigma$, $l \mapsto \g \,l$, and $u \mapsto 2nu$ yields
\begin{align}
\mathbb{E}_n	&= \frac{12^n}{\lm \T_n \rm} \sum_{\s\in\Sg} n^{\frac{|E(\s)|-g}2}\, 2^{\frac{|E(\s)|-3g +2}2}\, 3^g
		\int_{\mathcal{S}^\s} d\mathcal{L}^\s \ 
		A_n^\s \prod_{\e\in{\vec E}(\s)} B_n^{\s,\e} \prod_{\e\in\ori{E}(\s)} C_n^{\s,\e}\label{intg}
\end{align}
where
\begin{align*}
A_n^\s		&= \1{\mathcal{E}_n^\s}(nm,\sqrt{2n}\,\sigma,2n\,u)\ 
		\varphi\bigg( \s,\frac{2\lf n m \rf+ \lfloor\sqrt{2n}\,\sigma\rfloor}{2n},\frac{\lfloor\sqrt{2n}\,\sigma\rfloor}{\sqrt{2n}},
					\frac{\lfloor\g l\rfloor} \g,\frac {\lf 2nu \rf} {2n}\bigg),\\
B_n^{\s,\e}	&= n\, h_1\lp{\lf nm \rf}^\e, \lfloor \sqrt{2n}\,\sigma \rfloor^\e\rp,\\
C_n^{\s,\e}	&= \g\,h_2\Big(\lfloor \sqrt{2n}\,\sigma \rfloor^\e,\lfloor \g l \rfloor^{\e^+}\!\!- {\lfloor \g l \rfloor}^{\e^-}\Big).   		
\end{align*}

\bigskip

2) We are now going to see that every integral term of the sum appearing in the equation \eqref{intg} converges, by dominated convergence. We no longer use (\theeqit.1') but (\theeqit.1) in the identification (\theeqit). Because
$$\frac {2 \lf nm \rf^{\e_*}\!(n)+\lfloor \sqrt{2n} \, \sigma^\e \rfloor}{2n} = 1 - \hspace{-3mm} \sum_{\e\in \vec E(\s)\bs\{\e_*\}}\hspace{-3mm} \frac {2 \lfloor nm^\e \rfloor + \lfloor \sqrt{2n} \, \sigma^\e \rfloor}{2n}
			\ton 1 - \hspace{-3mm} \sum_{\e\in \vec E(\s)\bs\{\e_*\}} \hspace{-3mm} m^\e = m^{\e_*},$$
we see that
$A_n^\s \to \1{\lb m^{\e_*}\!\ge\, 0,\,u< m^{\e_*}\rb}\ \varphi\lp \s, m,\sigma, l,u \rp$. Thanks to Proposition \ref{petrov}, we then obtain
$$B_n^{\s,\e} \to -\ga'_{m^\e}(\sigma^\e) \sand C_n^{\s,\e} \to \ga_{\sigma^\e}(l^\e).$$

It remains to prove that the convergence is dominated. To that end, we use the second part of Proposition \ref{petrov}. In the remainder of the proof, $C$ will denote a constant in $(0,\infty)$, the value of which may differ from line to line. First, notice that
$$A_n^\s \le \| \varphi \|_\infty.$$
Then, applying Proposition \ref{petrov} with $r =3$, we obtain, for $n \ge 2$,
\small
\begin{align*}
B_n^{\s,\e} &= \frac {2n \, \lfloor\sqrt{2n}\,\sigma\rfloor^\e} {\big( 2{\lf nm \rf}^\e+{\lfloor\sqrt{2n}\,\sigma\rfloor}^\e\big)^{\frac32}}\ 
		\frac12 \big( 2{\lf nm \rf}^\e+{\lfloor\sqrt{2n}\,\sigma\rfloor}^\e \big)^{\frac12}\,
		Q^S_{2{\lf nm \rf}^\e+{\lfloor\sqrt{2n}\,\sigma\rfloor}^\e} \big( \lfloor\sqrt{2n}\,\sigma\rfloor^\e\big)\ 
		\1{\lb \sqrt{2n}\,\sigma^\e\ge 1\rb}\\
	&\le C \, \lp \frac {\lfloor\sqrt{2n}\sigma\rfloor^\e} {\sqrt {2n}} \rp^{-2}
		\frac {\lp \frac {\lp\lfloor\sqrt{2n}\sigma\rfloor^\e\rp^2} {2 \lf n m \rf^\e + \lfloor\sqrt{2n}\sigma\rfloor^\e}\rp^{3/2}}
		  {1 + \lp \frac {\lp\lfloor\sqrt{2n}\sigma\rfloor^\e\rp^2} {2 \lf n m \rf^\e + \lfloor\sqrt{2n}\sigma\rfloor^\e}\rp^{3/2}}\ \1{\lb \sqrt{2n}\,\sigma^\e\ge 1\rb}\\
%
%
	& \le C \, (m^\e)^{-1} \wedge \lp \sigma^\e \rp^{-2},
\end{align*}
\normalsize
where we used the fact that for $x \ge 1$, $\lf x \rf^{-1} \le 2 /x$. The case $\lf nm \rf =0$ is to be treated separately, and is left to the reader. Applying now Proposition \ref{petrov} with $r =2$, we find that, for $n \ge 2$,
\begin{align*}
C_n^{\s,\e} &= \frac {(2n)^{\frac 14}}{\big({\lfloor\sqrt{2n}\sigma\rfloor}^\e \big)^{\frac12}}\ 
	\sqrt{\frac23}\big({\lfloor\sqrt{2n}\sigma\rfloor}^\e \big)^{\frac12} \, Q^M_{{\lfloor\sqrt{2n}\sigma\rfloor}^\e}\Big({\lfloor\g l\rfloor}^{\e^+}\!\!-{\lfloor\g l\rfloor}^{\e^-}\Big)\ \1{\lb \sqrt{2n}\,\sigma^\e\ge 1\rb}\\
	&\le \frac C{\sqrt {\sigma^\e}} \lp{1+ \frac 3 2 \frac {\Big({\lfloor\g l\rfloor}^{\e^+}\!\!-{\lfloor\g l\rfloor}^{\e^-} \Big)^2} {{\lfloor\sqrt{2n}\sigma\rfloor}^\e}}\rp^{-1}\\
	&\le \frac C{\sqrt {\sigma^\e}} \lp{1+ \frac {\big( \big| l^{\e^+}\!\!-l^{\e^-} \big| - 1 \big)^2 } {\sigma^\e}\1{\lb\lm l^{\e^+}\!-l^{\e^-} \rm > 1\rb}}\rp^{-1}.
\end{align*}

Any integrand in the equation \eqref{intg} is then bounded by
\begin{equation}
\label{integrand}
C \prod_{\e\in\vec{E}(\s)} (m^\e)^{-1} \wedge \lp \sigma^\e \rp^{-2} \prod_{\e\in\ori{E}(\s)}\lp \sigma^\e \rp^{-1/2} \lp{1+ \frac {\big( \big| l^{\e^+}\!\!-l^{\e^-} \big| - 1 \big)^2 } {\sigma^\e}\1{\lb\lm l^{\e^+}\!-l^{\e^-} \rm > 1\rb}}\rp^{-1}.
\end{equation}
We have to see that this expression is integrable. First, note that we integrate with respect to~$u$ on a compact set.
Moreover,
\begin{align*}
\int_{\R} dl^{\e^-}\lp{1+ \frac {\big( \big| l^{\e^+}\!\!-l^{\e^-} \big| - 1 \big)^2 } {\sigma^\e}\1{\lb\lm l^{\e^+}\!-l^{\e^-} \rm > 1\rb}}\rp^{-1} 
		&= 2 + \pi \sqrt{\sigma^\e}\\ &\le C \, 1 \vee \sqrt {\sigma^\e},
\end{align*}
and we have the same bound if we integrate with respect to $l^{\e^+}$ instead of $l^{\e^-}$.

It is possible to injectively associate with every vertex $v\in V(\s)\bs\{\e_*^-\}$ a half-edge $\e_v\in \ori E(\s)$ such that $v$ is an extremity of $\e_v$. Let us call $E_V$ the range of such an injection.
The integral of the expression \eqref{integrand} with respect to $u$ and $l$ is then bounded by
$$C \prod_{\e\in\vec{E}(\s)} (m^\e)^{-1} \wedge \lp \sigma^\e \rp^{-2}
		\prod_{\e\in E_V} 1 \vee \lp \sigma^\e \rp^{-1/2}
		\!\prod_{\e\in\ori{E}(\s)\bs E_V}\!\! \lp \sigma^\e \rp^{-1/2}.$$
Finally, it is easy to see that this expression, once integrated\footnote{Be careful that, when integrating with respect to $\sigma^\e$ for some $\e\in\ori E(\s)$, both half-edges $\e$ and $\bar\e$ are to be considered.} with respect to $\sigma$, is bounded by $C \prod_{\e\in\vec{E}(\s)} (m^\e)^{-7/8}$, which is integrable with respect to $m$.

\bigskip

3) We just saw that the integral expression in \eqref{intg} converges toward
$$\int_{\mathcal{S}^\s} d\mathcal{L}^\s \ 
		\1{\lb m^{\e_*}\!\ge\, 0,\,u< m^{\e_*}\rb}\ \varphi\lp \s, m,\sigma, l,u \rp
		\prod_{\e\in{\vec E}(\s)} - \ga'_{m^\e} \lp \sigma^\e \rp
		\prod_{\e\in\ori{E}(\s)} \ga_{\sigma^\e} \lp l^\e \rp.$$
		
The dominant terms in the equation \eqref{intg} are the ones for which $|E(\s)|$ is the largest. The corresponding schemes are exactly the dominant ones: for a scheme, $2\,|E(\s)| = \sum_{v\in V(\s)} \text{deg}(v) \ge 3\,|V(\s)|$ and the Euler characteristic formula gives $|E(\s)|\le 6g-3$, the equality being reached when $2\,|E(\s)| = 3\,|V(\s)|$, that is when $\s$ is dominant. Note that this situation is exactly the same as the one encountered in \cite{chapuy08sum,chapuy07brm,miermont09trm}.

Hence, if $\varphi$ is momentarily chosen to be constantly equal to $1$, we obtain that
\begin{equation}\label{tg2}
\lm \T_n \rm \sim 12^n \, n^{\frac{5g-3}2}\, 2^{\frac{3g -1}2}\, 3^g \ \Upsilon
\end{equation}
where $\Upsilon$ is defined by \eqref{upsi}. Finally,
$$\mathbb{E}_n \ton \frac1\Upsilon \sum_{\s\in\Sg^*} \int_{\mathcal{S}^\s} d\mathcal{L}^\s \ 
		\1{\lb m^{\e_*}\ge 0,\,u< m^{\e_*}\rb}\ \varphi\lp \s, m,\sigma, l,u \rp
		\prod_{\e\in{\vec E}(\s)} - \ga'_{m^\e} \lp \sigma^\e \rp
		\prod_{\e\in\ori{E}(\s)} \ga_{\sigma^\e} \lp l^\e \rp,$$
which is the result we sought.
\end{pre}


\section{Convergence of the Motzkin bridges and the forests}\label{secmf}

Conditionally given the vector
$$\lp \s_n, \lp m_{n}^\e\rp_{\e\in \vec E(\s_n)}, \lp\sigma_{n}^\e\rp_{\e\in \vec E(\s_n)}, \lp l_{n}^v\rp_{v\in V(\s_n)} \rp,$$
the Motzkin bridges $\mM_n^\e$, $\e\in\ori E (\s_n)$ and the well-labeled forests $(\f_n^\e,\l_n^\e)$, $\e \in \vec E(\s_n)$ are independent and
\begin{itemize}
	\item for every $\e\in \ori E(\s_n)$, $\mM_n^\e$ is uniformly distributed over the set $\M_{[0,\sigma_n^\e]}^{0\to l_n^\e}$ of Motzkin bridges on $[0,\sigma_n^\e]$ from $0$ to $l_n^\e = l_n^{\e^+} - l_n^{\e^-}$,
	\item for every $\e\in \vec E(\s_n)$, $(\f_n^\e,\l_n^\e)$ is uniformly distributed over the set $\fF_{\sigma_n^\e}^{m_n^\e}$ of well-labeled forests with $\sigma_n^\e$ trees and $m_n^\e$ tree edges. 
\end{itemize}

The convergence of Motzkin bridges is already known. We will properly state the result we need in Lemma \ref{bn}.

The convergence of a uniform well-labeled tree with $n$ edges is well-known, see \cite{chassaing04rpl}, for example. We will need a conditioned version of this result: roughly speaking, instead of looking at one large tree with~$n$ edges uniformly labeled such that the root label is $0$, we look at a forest with~$n$ edges, a number of trees growing like $\sqrt n$, that are uniformly labeled provided the root label of every tree is $0$.
For that matter, we will adapt the arguments provided in \cite[Chapter 6]{legall94mbp}.

\bigskip

Let us define the space $\K$ of continuous real-valued functions on $\R_+$ killed at some time:
$$\K \de \bigcup_{x \in \R_+} \C([0, x],\R).$$

For an element $f \in \K$, we will define its lifetime $\sigma(f)$ as the only $x$ such that $f \in \C([0, x],\R)$. We endow this space with the following metric:
$$d_\K (f,g) \de |\sigma(f) - \sigma(g) | + \sup_{y \ge 0} \lm f\big(y\wedge \sigma(f)\big)-g\big(y\wedge\sigma(g)\big)\rm.$$

\bigskip

Recall that we use the notation $\underline{X}(s)$ for the infimum up to time $s$ of any process $X \in \K$. Throughout this section, $m$ and $\sigma$ will denote positive real numbers and $l$ will be any real number.

\subsection{Brownian bridge and first-passage Brownian bridge}

The results we show in this section are part of the probabilistic folklore. Because of the scarceness of the references, we will give complete proofs for the sake of self-containment.

We define here the Brownian bridge $B_{[0,m]}^{0\to l}$ on $[0,m]$ from $0$ to $l$ and the first-passage Brownian bridge $F_{[0,m]}^{0\to -\sigma}$ on $[0,m]$ from $0$ to $-\sigma$. Informally, $B_{[0,m]}^{0 \to l}$ and $F_{[0,m]}^{0 \to -\sigma}$ are a standard Brownian motion $\beta$ on $[0,m]$ conditioned respectively on the event $\{\beta_m = l\}$ and on hitting $-\sigma$ for the first time at time $m$. Of course, both theses events occur with probability $0$ so we need to define these objects properly. There are several equivalent ways do do so (see for example \cite{billingsley68cpm, revuz99cma, bertoin03ptf}).

Remember that we call $\ga_a$ the Gaussian density with variance $a$ and mean $0$, as well as $\ga_a'$ its derivative. Let $(\beta_t)_{0\le t\le m}$ be a standard Brownian motion. As explained in \cite[Proposition~1]{fitzsimmons93mbc}, the law of the Brownian bridge is characterized by $B_{[0,m]}^{0\to l}(m) = l$ and the formula
\begin{equation}\label{b}
\E{f \lp \lp B_{[0,m]}^{0\to l}(t)\rp_{0 \le t \le m'} \rp} = \E{f \lp \lp\beta_t\rp_{0 \le t \le m'} \rp \frac {\ga_{m - m'}(l-\beta_{m'})}{\ga_m(l)}}
\end{equation}
for all bounded measurable function $f$ on $\K$, for all $0\le m' < m$.

We define the law of the first-passage Brownian bridge in a similar way, by letting
\begin{equation}\label{f}
\E{f \lp \lp F_{[0,m]}^{0\to-\sigma}(t)\rp_{0 \le t \le m'} \rp} = \E{f \lp \lp \beta_t\rp_{0 \le t \le m'} \rp \frac {\ga_{m - m'}'(-\sigma - \beta_{m'})}{\ga_m'(-\sigma)}\,
\1{\lb {\rlap{\scriptsize \hspace{0.2mm}{\underline{\hspace{1.2mm}}}}\beta}_{m'} > - \sigma\rb}}
\end{equation}
for all bounded measurable function $f$ on $\K$, for all $0\le m' < m$, and
$$F_{[0,m]}^{0\to -\sigma}(m) = -\sigma.$$

These formulae set the finite-dimensional laws of the first-passage Brownian bridge. It remains to see that it admits a continuous version. Because its law is absolutely continuous with respect to the Wiener measure on every $[0,m']$, $m' < m$, the only problem arises at time $m$. We will, however, use the Kolmogorov lemma \cite[Theorem 1.8]{revuz99cma} to obtain the continuity of the whole trajectory. We will see during the proof of Lemma \ref{fn} that, as for the Brownian motion, the trajectories of the first-passage bridge are $\alpha$-H\"older for every $\alpha < 1/2$.

\bigskip
The motivation of these definitions may be found in the following lemma:

\begin{lem}
Let $(\beta_t)_{0\le t \le m}$ be a standard Brownian motion. Let $(B^\epsilon_t)_{0\le t \le m}$ and $(F^\epsilon_t)_{0\le t \le m}$ have the law of $\beta$ conditioned respectively on the events
$$\lb |\beta_m - l| < \eps \rb \sand \lb \beta_m < - \sigma + \eps,\, {\rlap{\hspace{0.3mm}\underline{\hspace{1.4mm}}}\beta}_{m} > - \sigma - \eps \rb.$$

Then, as $\epsilon$ goes to $0$,
$$B^\eps \to B_{[0,m]}^{0\to l} \sand F^\eps \to F_{[0,m]}^{0\to -\sigma}$$
in law in the space $\lp\C([0,m], \R), \| \cdot \|_\infty \rp$.
\end{lem}

The proof of this lemma uses similar methods as those we will use for Lemma \ref{bn} so we let the details to the reader. In what follows, we will use the following lemma, which is a consequence of the Rosenthal Inequality \cite[Theorem 2.9 and 2.10]{petrov95ltp}:

\begin{lem}\label{rosenthal}
Let $X_1$, $X_2$,\dots $X_k$ be independent centered random variables and $q \ge 2$. Then, there exists a constant $c(q)$ depending only on $q$ such that 
$$\E{\lm \sum_{i=1}^k X_i \rm^{q}} \le c(q) \, k^{\frac q 2 -1} \sum_{i=1}^k \E{\lm X_i \rm^{q}}.$$

In particular, if $X_1$, $X_2$,\dots $X_k$ are i.i.d.,
$$\E{\lm \sum_{i=1}^k X_i \rm^{q}} \le c(q) \, k^{\frac q 2} \,\E{\lm X_1 \rm^{q}}.$$
\end{lem}

\subsubsection*{Discrete bridges}

We will see in this paragraph two lemmas showing that these two objects are the limits of their discrete analogs. These lemmas, in themselves, motivate our definitions of bridges and first-passage bridges. Let us begin with bridges.

We consider a sequence $(X_k)_{k \ge 0}$ of i.i.d. centered integer-valued random variables with a moment of order $q_0$ for some $q_0\geq 3$. We write $\eta^2 \de \var(X_1)$ its variance and $h$ its maximal span.
We define $\Sigma_i \de \sum_{k=0}^i X_k$ and still write $\Sigma$ its linearly interpolated version. Let $(m_n)\in \Zp^\N$ and $(l_n) \in \Z^\N$ be two sequences of integers such that 
$$m_{(n)} \de \frac{m_n}{n} \ton m \sand l_{(n)} \de \frac{l_n}{\eta \sqrt{n}} \ton l.$$

Let $(B_n(i))_{0\le i \le m_n}$ be the process whose law is the law of $(\Sigma_i)_{0\le i \le m_n}$ conditioned on the event
$$\{ \Sigma_{m_n}=l_n \},$$
which we suppose occurs with positive probability. We let
$$B_{(n)} \de \lp \frac {B_n(ns)}{\eta \sqrt{n}} \rp_{0\le s \le m_{(n)}}$$
be its rescaled version.

\begin{lem}\label{bn}
As $n$ goes to infinity, the process $B_{(n)}$ converges in law toward the process $B^{0 \to l}_{[0,m]}$, in the space $\lp \K, \, d_\K \rp$.
\end{lem}

\begin{pre}
We note $\mathcal{F}_i \de \sigma( \Sigma_k, 0 \le k \le i)$ the natural filtration associated with $\Sigma$.
Applying the Skorokhod theorem, we may and will assume that 
$$\lp \frac {\Sigma_{ns}} {\eta \sqrt {n}} \rp_{0\le s\le m}$$
converges a.s. toward a standard Brownian motion $(\beta_s)_{0\le s\le m}$ for the uniform topology.

\medskip

1) Let $m' < m$. We begin by looking at $B_{(n)}$ on $[0,m']$. For $n$ large enough, $\lc nm' \rc < m_n$. Let $f$ be continuous bounded from $\K$ to $\R$. We have
\begin{align}
\E{f \lp (B_{(n)}(s))_{0 \le s \le m'} \rp} & = \E{f \lp \lp \frac {\Sigma_{ns}} {\eta \sqrt {n}} \rp_{0 \le s \le m'} \rp \,\Big| \ \Sigma_{m_n} = l_n}\notag\\
&= \E{f \lp \lp \frac {\Sigma_{ns}} {\eta \sqrt {n}} \rp_{0 \le s \le m'} \rp \frac
	{\P\lp \Sigma_{m_n} = l_n \, \lm \ \mathcal{F}_{\lc nm' \rc} \rno  \rp}
	{\P\lp \Sigma_{m_n} = l_n \rp}}.\label{bb}
\end{align}
Recall the notation $Q^\Sigma_{k} (i)=\P\lp \Sigma_{k} = i \rp$. Using the Markov property, we obtain
\begin{align}
\P\lp \Sigma_{m_n} = l_n \, \lm \ \mathcal{F}_{\lc nm' \rc} \rno \rp &= Q^\Sigma_{m_n - \lc n m' \rc} (l_n -\Sigma_{\lc nm' \rc}) \notag\\
&\sim \frac h {\eta\sqrt n} \, \ga_{m-m'}(l -\beta_{m'}).\label{eqbb}
\end{align}
where the second line comes from Proposition \ref{petrov}. Note that the denominator of the fractional term in \eqref{bb} is the same as the numerator when $m'$ is chosen to be $0$.
So the fractional term in \eqref{bb} converges a.s. toward
$$\frac{\ga_{m-m'}(l -\beta_{m'})}{\ga_{m}(l)},$$
the convergence being dominated---by Proposition \ref{petrov}. Finally,
\begin{align*}
\E{f \lp (B_{(n)}(s))_{0 \le s \le m'} \rp}
		&\ton \E{f \lp (\beta_s)_{0 \le s \le m'} \rp \frac {\ga_{m - m'}(l - \beta_{m'})}{\ga_m(l)}}\\
		&\hspace{7mm} = \E{f \lp (B_{[0,m]}^{0\to l}(s))_{0 \le s \le m'} \rp}.
\end{align*}

\medskip
2) We will use the following lemmas, the proofs of which we postpone right after the end of this proof.

\begin{lem}\label{lembn}
There exists an integer $n_0 \in \N$ such that, for every $2 \le q \le q_0$, there exists a constant $C_q$ satisfying, for all $n \ge n_0$ and $0 \le s \le t \le m_{(n)}$,
$$\E{\lm B_{(n)}(t) - B_{(n)}(s) \rm^{q}} \le C_q \lm t -s \rm^{\frac q 2}.$$
\end{lem}

\begin{lem}\label{lemb}
We note $B\de B_{[0,m]}^{0\to l}$. For any $q \ge 2$, there exists a constant $C_q$ such that, for all $0 \le s \le t \le m$,
$$\E{\lm B(t) - B(s) \rm^{q}} \le C_q \lm t -s \rm^{\frac q 2}.$$
\end{lem}

By the Portmanteau theorem \cite[Theorem 2.1]{billingsley68cpm}, we can restrict ourselves to bounded uniformly continuous functions from $\K$ to $\R$. Let $f$ be such a function. Let $\eps >0$, and $\delta >0$ be such that $d_\K(X,Y) < \delta$ implies $|f(X)-f(Y)| < \eps$.

Let $0 < \alpha < 1/2-1/q_0$. Thanks to Lemmas \ref{lembn} and \ref{lemb}, Kolmogorov's criterion \cite[Theorem~3.3.16]{stroock99pta} provides us with some constant $C$ such that
$$\sup_n \P\lp B_{(n)} \notin K\rp \vee \P\lp B \notin K\rp < \frac \eps{\|f\|_\infty},$$
where
$$K \de \lb X \in \K\,:\, \sup_{s\neq t} \frac {\lm X(t) - X(s)\rm}{\lm t-s\rm^\alpha}\le C \rb.$$

%
%

We take $m'$ satisfying
$$|m-m'| + C\,|m-m'|^\alpha < \frac \delta 2,$$
so that, for $n$ sufficiently large,
$$|m_{(n)}-m'| + C\,|m_{(n)}-m'|^\alpha < \delta.$$

For any function $X = (X(s))_{0\le s \le x} \in \K$, we define $X_{|y} \de (X(s))_{0\le s \le y}\in \K$. Hence
\begin{align}
\E{\lm f\lp B_{(n)}\rp - f \lp B \rp \rm} \le \E{\lm f\lp B_{(n)}\rp - f\lp {B_{(n)}}_{|m'} \rp \rm} + \E{\lm f\lp {B_{(n)}}_{|m'} \rp - f\lp B_{|m'} \rp \rm}\notag\\
&\hspace{-3.2cm} + \E{\lm f\lp B_{|m'} \rp - f \lp B \rp \rm}\!.\label{cvb}
\end{align}
Thanks to point 1), for $n$ large enough, the second term of the right-hand side of \eqref{cvb} is less than~$\eps$. The first and third terms are treated in the same way (for the third term, just remove the $(n)$'s): on the set $\lb B_{(n)} \in K \rb$,
\begin{align*}
d_{\K} \lp B_{(n)},{B_{(n)}}_{|m'} \rp &= |m_{(n)}-m'| + \sup_{m'\le t \le m_{(n)}} \lm B_{(n)}(t)-B_{(n)}(m') \rm\\
				       &\le |m_{(n)}-m'| + C\,|m_{(n)}-m'|^\alpha\\
				       &< \delta,
\end{align*}
and
\begin{align*}
\E{\lm f\lp B_{(n)}\rp - f\lp {B_{(n)}}_{|m'} \rp \rm} &\le \E{\lm f\lp B_{(n)}\rp - f\lp {B_{(n)}}_{|m'} \rp \rm \, \1{\lb B_{(n)} \in K \rb}} +2 \|f\|_\infty\, \P\lp B_{(n)} \notin K\rp\\
							&< 3\eps.
\end{align*}

All in all, for $n$ large enough
$$\E{\lm f\lp B_{(n)}\rp - f \lp B \rp \rm} \le 7\eps,$$
and $B_{(n)}$ converges weakly toward $B$.
\end{pre}

It remains to prove Lemmas \ref{lembn} and \ref{lemb}.

\begin{preuv}{of Lemma \ref{lembn}}
If $|t-s| < 1/n$, the fact that $B_n$ is linear on every interval $[i,i+1]$ implies that
$$|B_{(n)}(t) - B_{(n)}(s)|\le \frac {n(t-s)}{\eta\sqrt n} \le \frac 1\eta \sqrt{t-s},$$
which gives the desired result. By the triangular inequality, we can restrict ourselves to the cases where $ns$ and $nt$ are integers, and either $t\le m_{(n)}/2$ or $m_{(n)}/2 \le s$.

First, let us suppose that $0\le s\le t\le m_{(n)}/2$. Applying \eqref{bb} with $m'=t$ and the proper function $f$, we obtain
\begin{equation}\label{bnq}
\E{\lm B_{(n)}(t) - B_{(n)}(s) \rm^{q}} = \eta^{-q} n^{-\frac q 2} \, \E{\lm \Sigma_{nt} - \Sigma_{ns} \rm^{q} \frac 
			{Q^\Sigma_{m_n-nt} (l_n - \Sigma_{nt})}{Q^\Sigma_{m_n}(l_n)}}.
\end{equation}

The asymptotic formula \eqref{eqbb} and the fact that $m_{(n)} \to m$ yield the existence of a positive constant $c$ and an integer $n_0$ such that for $n \ge n_0$,
$$\sqrt n \, Q^\Sigma_{m_n}(l_n) \ge c \sand m_{(n)} > \frac m 2.$$
Then Proposition \ref{petrov} ensures us that for $n \ge n_0$,
\begin{align*}
\sqrt n\, Q^\Sigma_{m_n-nt} (l_n - \Sigma_{nt})
	&\le \sqrt n \, \sup_{x \in \R}\,\sup_{y > \frac m 4} Q^\Sigma_{ny}\lp x\sqrt n\rp\\
	&\le \frac {2h}{\eta\sqrt m}\, \sup_{x \in \R}\,\sup_{y > 0}\,\sup_{n\in \N} \lp \frac{\eta}{h} \sqrt{ny}\, Q^\Sigma_{ny}\lp x\sqrt n\rp\rp
<\infty.
\end{align*}
Thus, the fractional term in the equation \eqref{bnq} is uniformly bounded as soon as $n \ge n_0$, and 
\begin{align*}
\E{\lm B_{(n)}(t) - B_{(n)}(s) \rm^{q}} &\le C \, n^{-\frac q 2}\, \E{\lm \Sigma_{nt} - \Sigma_{ns} \rm^{q}}\\
				      &\le C \, n^{-\frac q 2}\, \E{\lm \Sigma_{n(t-s)} \rm^{q}}\\
				      &\le C_q \lm t - s\rm^{\frac q 2}
\end{align*}
by means of the Rosenthal Inequality (Lemma \ref{rosenthal}).

\medskip
Now, if $m_{(n)}/2\le s\le t \le m_{(n)}$, we use the following time reversal invariance:
\begin{equation}\label{btpn}
\lp B_{(n)}(s)\rp_{0\le s\le m_{(n)}} \law \lp l_{(n)} - B_{(n)}(m_{(n)} - s)\rp_{0\le s\le m_{(n)}}.
\end{equation}
We have
$$\E{\lm B_{(n)}(t) - B_{(n)}(s) \rm^{q}} = \E{\lm B_{(n)}(m_{(n)} - s) - B_{(n)}(m_{(n)} -t) \rm^{q}}$$
and we are back in the case we just treated. Note that it is important that $m_{(n)}$ be a deterministic time.
\end{preuv}

\begin{preuv}{of Lemma \ref{lemb}}
We show the inequality for $2 \le q \le q_0$. As $B$ appears as the limit of $B_{(n)}$ (in a certain sense), we may choose the $X_k$'s to have arbitrarily large moments, and we see that it actually holds for any value of $q \ge 2$.
For $0\le s\le t < m$, point 1) in the proof of Lemma~\ref{bn} shows that
$$\lp B_{(n)}(s),B_{(n)}(t) \rp \tol (B(s), B(t)),$$
and 
\begin{align*}
\E{\lm B(t) - B(s) \rm^{q}} & = \lim_{M \to \infty} \E{\lm B(t) - B(s) \rm^{q}\wedge M}\\
			    & = \lim_{M \to \infty} \lim_{n\to \infty} \E{\lm B_{(n)}(t) - B_{(n)}(s) \rm^{q}\wedge M}\\
			    &\le C_q \lm t -s \rm^{\frac q 2},
\end{align*}
where $C_q$ is the constant of Lemma \ref{lembn}. It only remains to see that $B_{(n)}(m\wedge m_{(n)}) \to B(m)$ in probability in order to obtain the same inequality for $t=m$. The time reversal invariance \eqref{btpn} implies that
$$B_{(n)}(m\wedge m_{(n)}) \law l_{(n)} - B_{(n)}\lp (m_{(n)}-m)\vee 0 \rp,$$
and, thanks to 1),
\begin{align*}
\lm B_{(n)}\lp (m_{(n)}-m)\vee 0 \rp \rm &\le \lm B_{(n)}\lp (m_{(n)}-m)\vee 0 \rp - B\lp (m_{(n)}-m)\vee 0 \rp \rm + \lm B\lp (m_{(n)}-m)\vee 0 \rp \rm\\
			  &\to 0
\end{align*}
in probability, so that $B_{(n)}(m\wedge m_{(n)}) \to l = B(m)$ in probability.
\end{preuv}

\subsubsection*{Discrete first-passage bridges}

We now see a lemma similar to Lemma \ref{bn} for first-passage bridges, in which we will only consider simple random walks.
Let $(m_n)\in \Zp^\N$ and $(\sigma_n) \in \N^\N$ be two sequences of integers such that 
$$m_{(n)} \de \frac{m_n}{n} \ton m \sand \sigma_{(n)} \de \frac{\sigma_n}{\sqrt{n}} \ton \sigma.$$

We consider a sequence $(X_k)_{k \ge 1}$ of i.i.d. random variables with law $(\delta_{-1} + \delta_1)/2$ and define $S_i \de \sum_{k=1}^i X_k$ (and, by convention, $S_0 =0$). We still write $S$ its linearly interpolated version. We call $(B_n(i))_{0\le i \le m_n}$ and $(F_n(i))_{0\le i \le m_n}$ the two processes whose laws are the law of $(S_i)_{0\le i \le m_n}$ conditioned respectively on the events
$$\{ S_{m_n}= -\sigma_n \} \sand \{ S_{m_n}= - \sigma_n,\, \underline S_{m_n-1} > - \sigma_n \},$$
which we suppose occur with positive probability.
Finally, we define
$$B_{(n)} \de \lp \frac {B_n(ns)}{\sqrt{n}} \rp_{0\le s \le m_{(n)}} \sand
	F_{(n)} \de \lp \frac {F_n(ns)}{\sqrt{n}} \rp_{0\le s \le m_{(n)}}$$
their rescaled versions.

\bigskip

There is actually a very convenient way to construct $F_n$ from $B_n$. For $0 \le k \le m_n$, the shifted path of $B_n$ is defined by
$$\varTheta_k(B_n)(x)=\lb \begin{array}{lcl}
                           B_n(k+x) - B_n(k)			& \text{ if} &\!\! 0\le x \le m_n-k,\\
                           B_n(k+x-m_n) + B_n(m_n) - B_n(k)	& \text{ if} &\!\! m_n-k\le x \le m_n.\\
                          \end{array}\rno$$
For $0\le k \le \sigma_n -1$, the first time at which $B_n$ reaches its minimum plus $k$ is noted
$$r_k(B_n) \de \inf \lb i\,:\, B_n(i) = \inf_{0\le j \le m_n}\!\! B_n(j) +k \rb.$$

The following proposition \cite[Theorem 1]{bertoin03ptf} gives a construction of $F_n$ from $B_n$.

\begin{prop}[Bertoin - Chaumont - Pitman]\label{bcp}
Let $\nu_n$ be a random variable independent of $S$ and uniformly distributed on $\{0,1,\dots,\sigma_n-1\}$. Then, the process $\varTheta_{r_{\nu_n}(B_n)}(B_n)$ has the same law as $F_n$.
\end{prop}

Using this construction, we may show that the first-passage Brownian bridge is the limit of its discrete analog:

\begin{lem}\label{fn}
As $n$ goes to infinity, the process $F_{(n)}$ converges in law toward the process $F^{0 \to -\sigma}_{[0,m]}$, in the space $\lp \K, \, d_\K \rp$.
\end{lem}

\begin{pre}
We begin as in the proof of Lemma \ref{bn}. We note $\mathcal{F}_i \de \sigma( S_k, 0 \le k \le i)$ the natural filtration associated with $S$, and by the Skorokhod theorem, we may and will assume that 
$$\lp \frac {S_{ns}} {\sqrt {n}} \rp_{0\le s\le m}$$
converges a.s. toward a standard Brownian motion $(\beta_s)_{0\le s\le m}$ for the uniform topology.

\medskip

1) Let $m' < m$. For $n$ large enough, $\lc nm' \rc < m_n$. Let $f$ be continuous bounded from $\K$ to~$\R$. We have
\begin{align}
\E{f \lp (F_{(n)}(s))_{0 \le s \le m'} \rp}\notag\\
&\hspace{-2cm} = \E{f \lp \lp \frac {S_{ns}} { \sqrt {n}} \rp_{0 \le s \le m'} \rp \,\Big| \ 
			S_{m_n} = - \sigma_n,\, \underline S_{m_n-1} > - \sigma_n}\notag\\
&\hspace{-2cm} = \E{f \lp \lp \frac {S_{ns}} { \sqrt {n}} \rp_{0 \le s \le m'} \rp \frac
	{\P\lp S_{m_n} = - \sigma_n,\, \underline S_{m_n-1} > - \sigma_n \, \lm \ \mathcal{F}_{\lc nm' \rc} \rno  \rp}
	{\P\lp S_{m_n} = - \sigma_n,\, \underline S_{m_n-1} > - \sigma_n \rp}}.\label{fpbb}
%
\end{align}
Recall the notation $Q_k^S(i) = \P(S_k=i)$. We have to deal with terms of the form
$$\P(S_k=-i,\, \underline S_{k-1} > -i) = \frac i k\, \P(S_k=-i) = \frac i k\, Q_k^S(-i),$$
where the first equality is an application of the so-called cycle lemma (see e.g. \cite[Lemma~2]{bertoin03ptf}). Using the Markov property, we obtain
\begin{align*}
\P\lp S_{m_n} = - \sigma_n,\, \underline S_{m_n-1} > - \sigma_n \, \lm \ \mathcal{F}_{\lc nm' \rc} \rno  \rp \\
%
%
&\hspace{-3cm}= \frac{\sigma_n + S_{\lc nm' \rc}}{m_n - \lc n m' \rc} \, Q^S_{m_n - \lc n m' \rc} \lp -\sigma_n - S_{\lc nm' \rc}\rp \1{\lb\underline{S}_{\lc nm' \rc} > - \sigma_n\rb}.
\end{align*}

Here again, the denominator of the fractional term in \eqref{fpbb} is the same as the numerator when~$m'$ is chosen to be $0$.
The fractional term in \eqref{fpbb} converges a.s. toward
$$\frac{\ga'_{m-m'}(-\sigma -\beta_{m'})}{\ga'_{m}(-\sigma)}\,\1{\lb {\rlap{\scriptsize \hspace{0.2mm}{\underline{\hspace{1.2mm}}}}\beta}_{m'} > - \sigma\rb},$$
and Proposition \ref{petrov} ensures that this convergence is dominated. So,
\begin{align*}
\E{f \lp (F_{(n)}(s))_{0 \le s \le m'} \rp} &\ton \E{f \lp (\beta_s)_{0 \le s \le m'} \rp \frac {\ga_{m - m'}'(-\sigma - \beta_{m'})}{\ga_m'(-\sigma)}\,
\1{\lb {\rlap{\scriptsize \hspace{0.2mm}{\underline{\hspace{1.2mm}}}}\beta}_{m'} > - \sigma\rb}}\\
		&\hspace{7mm} = \E{f \lp (F_{[0,m]}^{0\to-\sigma}(s))_{0 \le s \le m'} \rp}.
\end{align*}

\medskip
2) For any $\alpha > 0$ and $X = (X(s))_{0\le s \le x} \in \K$, we write
$$\|X\|_\alpha \de \sup_{0 \le s < t \le x} \frac{\lm X(t) - X(s) \rm}{|t-s|^\alpha}$$
its $\alpha$-H\"older norm. Proposition \ref{bcp} gives a stochastic domination of the $\alpha$-H\"older norm of $F_{(n)}$ by that of $B_{(n)}$: we may assume that $F_n = \varTheta_{r_{\nu_n}(B_n)}(B_n)$. If $0\le s < t \le m_{(n)} - r_{\nu_n}(B_n)$,
\begin{align*}
\lm F_{(n)}(t) - F_{(n)}(s) \rm &= \frac 1 {\sqrt n} \lm \varTheta_{r_{\nu_n}(B_n)}(B_n)(nt) - \varTheta_{r_{\nu_n}(B_n)}(B_n)(ns) \rm\\
				&= \frac 1 {\sqrt n} \lm B_n\lp r_{\nu_n}(B_n) + nt\rp - B_n\lp r_{\nu_n}(B_n) + ns\rp \rm\\
				&= \lm B_{(n)}\lp \frac {r_{\nu_n}(B_n)} n + t\rp - B_{(n)}\lp \frac {r_{\nu_n}(B_n)} n + s\rp \rm\\
				&\le \ln B_{(n)}\rn_\alpha \,\lm t-s \rm^\alpha.
\end{align*}

We obtain the same inequality when $m_{(n)} - r_{\nu_n}(B_n) \le s < t \le m_{(n)}$, and by the triangular inequality, we find
$$\ln F_{(n)}\rn_\alpha \le 2 \ln B_{(n)}\rn_\alpha.$$

\medskip
3) We now suppose that $0 < \alpha < 1/2$. Let $\eps>0$. Thanks to Lemma \ref{lembn}---for which we now have~$q_0$ arbitrarily large---and Kolmogorov's criterion, we can find some constant $C$ such that
\begin{equation}\label{kolfn}
\sup_n \P\lp F_{(n)} \notin K\rp < \eps \text{\hspace{5mm}with\hspace{5mm}} K \de \lb X \in \K\,:\, \|X\|_\alpha \le C \rb.
\end{equation}

Ascoli's theorem \cite[Chapter XX]{schwartz70atg} shows that $K$ is a compact set, so that the laws of the $F_{(n)}$'s are tight.

\medskip
4) We almost have the convergence of the finite-dimensional marginals of $F_{(n)}$ toward those of $F \de F_{[0,m]}^{0\to-\sigma}$. Point 1) shows that for any $p\ge 1$, $0\le s_1 < s_2 < \dots < s_p < m$,
$$\lp F_{(n)}(s_1),F_{(n)}(s_2),\dots,F_{(n)}(s_p) \rp \to \lp F(s_1),F(s_2),\dots,F(s_p) \rp.$$
It only remains to deal with the point $m$. Let $\delta > 0$. For $n$ large enough, on $\lb F_{(n)} \in K \rb$,
$$\lm F_{(n)}(m\wedge m_{(n)}) + \sigma \rm \le \lm \sigma_{(n)} - \sigma \rm + C\lm m_{(n)} - m \rm^\alpha < \delta,$$
therefore
$$\P\lp | F_{(n)}\lp m\wedge m_{(n)}\rp + \sigma | > \delta \rp \le \P\lp F_{(n)} \notin K \rp < \eps.$$

We have shown that $F_{(n)}\lp m\wedge m_{(n)}\rp$ converges in law toward the deterministic value $-\sigma$ so Slutzky's lemma allows us to conclude that the finite-dimensional marginals of $F_{(n)}$ converge toward those of $F$. This, together with the tightness of the laws of the $F_{(n)}$'s, yields the result thanks to Prokhorov's lemma.
\end{pre}

\bigskip

For any real numbers $m_1$, $m_2$, $l_1$, $l_2$, we define the bridge on $[m_1,m_2]$ from $l_1$ to $l_2$ by
$$\lp B_{[m_1,m_2]}^{l_1 \to l_2}(s)\rp_{m_1 \le s \le m_2} \de l_1 + \lp B_{[0,m_2-m_1]}^{0 \to l_2-l_1}(s-m_1)\rp_{m_1 \le s \le m_2},$$
and for $\sigma_1 > \sigma_2$, we define the first-passage bridge on $[m_1,m_2]$ from $\sigma_1$ to $\sigma_2$ by
$$\lp F_{[m_1,m_2]}^{\sigma_1 \to \sigma_2}(s)\rp_{m_1 \le s \le m_2} \de \sigma_1 + \lp F_{[0,m_2-m_1]}^{0 \to \sigma_2-\sigma_1}(s-m_1)\rp_{m_1 \le s \le m_2}.$$

\subsection{The Brownian snake}

We need a version of the Brownian snake's head driven by a first-passage Brownian bridge. There are several ways to define such an object.

We may define it as a the head of a Brownian snake with lifetime process a first-passage Brownian bridge $F_{[0,m]}^{\sigma \to 0}$ and starting from the path $0_\sigma \de t \in [0,\sigma] \mapsto 0$ (see \cite[Chapter IV]{legall99sbp} or \cite[Chapter 4]{duquesne02rtl} for a proper definition).

Let $(F_s)_{0\le s\le m}$ be a first-passage Brownian bridge from $\sigma$ to $0$. The Brownian snake driven by $F$ and started at $0_\sigma$ is the path-valued process $(F_s,(W(s,t),\, 0\le t \le F_s))_{0\le s \le m}$ whose law is defined by:
\begin{itemize}
 \item for all $0\le t \le \sigma$, $W(0,t) = 0$,
 \item for all $0\le s \le m$, $W(s,0) = 0$,
 \item the conditional law of $W(s,\cdot)$ given $F$ is the law of an inhomogeneous Markov process whose transition kernel is described as follows: for $0\le s \le s' \le m$,
	\begin{itemize}
	 \item $W(s',t) = W(s,t)$ for all $0\le t \le \inf_{[s,s']} F$,
	 \item $\lp W(s', \inf_{[s,s']} F + t) \rp_{0\le t \le F_{s'} - \inf_{[s,s']} F}$ is independent of $W(s,\cdot)$ and distributed as a real Brownian motion started from $W(s, \inf_{[s,s']} F)$ and stopped at time $F_{s'} - \inf_{[s,s']} F$.
	\end{itemize}
\end{itemize}

The head of this process is then defined by
$$\lp F_{[0,m]}^{\sigma \to 0}, Z_{[0,m]}\rp \de \lp \lp F_s \rp_{0\le s \le m}, \lp W(s, F_s)\rp_{0\le s \le m}\rp.$$

\bigskip
This description has the advantage of being very visual: $W(0,\cdot)$ is the function $0_\sigma$. Then, every time $F$ decreases, we erase the tip of the previous path, and when $F$ increases, we glue a part of an independent Brownian motion (see Figure \ref{0snake}).

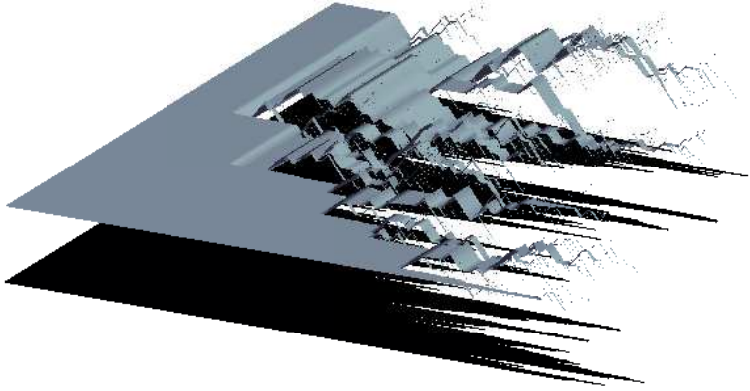
\sfig{0snake}{An approximation of the conditioned Brownian snake. The first-passage bridge from $\sigma$ to $0$ is represented by the shadowy part of the figure. In order to see $W(s,\cdot)$, one must ``cut'' the surface at $s$ and look at the edge of the cut piece.}{width=10cm}

In the following, we will only need the head and not the whole process. The following description gives a direct construction of this head. Conditionally given $F = F_{[0,m]}^{\sigma \to 0}$, we define a Gaussian process $(\Gamma_s)_{0\le s \le m}$ with covariance function
$$\cov(\Gamma_s, \Gamma_{s'}) = \inf_{[s,s']} (F - \underline{F}).$$
The processes $\lp F, \Gamma \rp$ then has the same law as the process $\big( F_{[0,m]}^{\sigma \to 0}, Z_{[0,m]}\big)$ defined above. 

\bigskip

We easily see that we can derive the law of the head from the law of the snake, and it is actually also possible to recover the whole snake from its head (see \cite[Section 2]{marckert03sss}): starting from the process $(F,Z)=\big( F_{[0,m]}^{\sigma \to 0}, Z_{[0,m]}\big)$, we define
$$W(s,t) \de Z\big( \inf\{r\ge s,\, F(r) = t \} \big), \hspace{1cm}0\le t\le F(s),\, 0\le s \le m.$$

The process $\big(F(s),(W(s,t),\, 0\le t \le F(s))\big)_{0\le s \le m}$ then has the law of the Brownian snake defined above. In particular, for $s \in [0,m]$ fixed, the process
$$\lp Z\big( \inf\{r\ge s,\, F(r) = t \} \big) \rp_{\underline F(s) \le t \le F(s)}$$
has the law of a real Brownian motion started from $0$. Using time reversal invariance, we see that the process
$$\Big( Z\big( \inf\{r\ge s,\, F(r) = F(s) - x \} \big) - Z(s) \Big)_{0 \le x \le F(s) - \underline F(s)}$$
has the same law. This fact will be used in Section \ref{secpro}.

\subsection{The discrete snake}\label{snhe}

We will describe here an analog of the Brownian snake in the discrete setting. Let us first consider three sequences of integers $(\sigma_n)$, $(m_n)$ and $(l_n)$ such that
$$\sigma_{(n)} \de \frac{\sigma_n}{\sqrt{2n}} \to \sigma,\hspace{1cm} m_{(n)} \de \frac{2m_n+\sigma_n}{2n} \to m \sand
l_{(n)} \de \frac{l_n}{\gamma n^{\frac14}} \to l.$$

We call $(C_n, L_n)$ the contour pair of a random forest uniformly distributed over the set $\fF_{\sigma_n}^{m_n}$ of well-labeled forests with $\sigma_n$ trees and $m_n$ tree edges. We define
$$C_{(n)} \de \lp \frac {C_n(2nt)} {\sqrt{2n}} \rp_{0\le t \le m_{(n)}} \sand L_{(n)} \de \lp \frac {L_n(2nt)} {\g}\rp_{0\le t \le m_{(n)}}$$
their scaled versions.

We define the discrete snake $\lp W_n(i,j), 0 \le j \le C_n(i) \rp_{0 \le i \le 2m_n +\sigma_n}$ by (see Figure \ref{dsnake})
$$W_n(i,j) \de L_n \lp \sup \lb k \le i \ : \ C_n(k)=j \rb \rp = L_n \lp \inf \lb k \ge i \ : \ C_n(k)=j \rb \rp.$$
Let $(\f,\l)$ be the well-labeled forest coded by $(C_n, L_n)$. Then for $0 \le i \le 2m_n +\sigma_n$,
$$\lp W_n(i,j) \rp_{0 \le j \le C_n(i)}$$
records the labels of the unique path going from $t(\f)+1$ to $\f(i)$. As a result, $W_n(i,j)=0$ for $0 \le j \le t(\f)+1-\a(\f(i))$.

\stfig{dsnake}{Discrete snake}{width=10cm}{
\psfrag{c}[][]{\textcolor{red}{$C_n$}}
\psfrag{i}[][][.8]{$i$}
\psfrag{j}[][][.8]{$j$}
}

We then extend $W_n$ to $\{ (s,t)\,:\, s \in [0,2m_n+\sigma_n], \, t\in [0,C_n(s)] \}$ by linear interpolation and we let, for $0\le s \le m_{(n)}$, $0\le t \le C_{(n)}(s)$,
$$W_{(n)}(s,t) \de \frac {W_n(2ns, \sqrt{2n}\, t)}{\g}.$$

For each $0\le s \le m_{(n)}$, $W_{(n)}(s,\cdot)$ is a path lying in
$$\K_0 \de \lb f \in \K \ | \ f(0) = 0 \rb,$$
so that we can see $W_{(n)}$ as an element of
$$\W_0 \de \bigcup_{x \in \R_+} \C([0, x],\K_0).$$

For $X \in \W_0$, we call $\xi(X)$ the real number such that $X \in \C([0,\xi(X)],\K_0)$, and we endow~$\W_0$ with the metric
$$d_{\W_0}(X,Y) \de |\xi(X) - \xi(Y)| + \sup_{s\ge 0} d_\K\lp X(s\wedge \xi(X),\cdot),Y(s\wedge \xi(Y),\cdot) \rp.$$

\subsection{Convergence of a uniform well-labeled forest}

We will prove the following result.

\begin{prop}\label{wsnake}
The pair $(C_{(n)}, W_{(n)})$ converges weakly toward the pair $\lp F^{\sigma \to 0}_{[0,m]}, W\rp$, in the space $\lp \K, \, d_\K \rp \times \lp \W_0, \, d_{\W_0} \rp$.
\end{prop}

We readily obtain the following corollary:

\begin{corol}\label{snak}
The pair $(C_{(n)}, L_{(n)})$ converges weakly toward the pair $\lp F^{\sigma \to 0}_{[0,m]}, Z_{[0,m]}\rp$, in the space $\lp \K, \, d_\K \rp^2$.
\end{corol}

Proposition \ref{wsnake} may appear stronger than Corollary \ref{snak}, but is actually not, because of the strong link between the whole snake and its head \cite{marckert03sss}. We begin by a lemma.

\begin{lem}\label{tight}
For all $0 < \delta < 1/4$, for all $\eps >0$, there exist a constant $C$ and an integer $n_0$ such that, as soon as $n\ge n_0$, $\P(W_{(n)} \notin A) < \eps$, where
$$A\de \lb X \in \W_0\,:\, \sup_{s\neq s'} \frac {d_\K\lp X(s,\cdot) - X(s',\cdot) \rp}{\lm s-s'\rm^\delta} \le C \rb.$$
\end{lem}

\begin{pre}
It is based on \eqref{kolfn} and a similar inequality for Motzkin paths (which is merely Rosenthal Inequality). The fact that the steps of the random walks we consider are bounded allows us to take the $q$ of Lemma \ref{rosenthal} arbitrary large.

\medskip
Let $0\le s < s' \le m_{(n)}$.
Conditionally given $C_{(n)}$,
\begin{align*}
d_\K\lp W_{(n)}(s,\cdot),W_{(n)}(s',\cdot)\rp\\
&\hspace{-2cm}= \lm C_{(n)}(s) - C_{(n)}(s')\rm + \sup_{t\ge a_n} \lm W_{(n)}\lp s,t\wedge C_{(n)}(s)\rp - W_{(n)}\lp s',t\wedge C_{(n)}(s')\rp \rm,
\end{align*}
where $a_n \de \inf_{[s,s']} C_{(n)}$.

\bigskip
We need to distinguish two cases: 
\begin{itemize}
 \item if $b_n \de \inf_{[0,s]} C_{(n)} \le a_n$, then
	$$\lp W_{(n)}(s,t) - W_{(n)}(s,a_n) \rp_{a_n \le t \le C_{(n)}(s)}$$
	is merely a rescaled Motzkin path.
 \item if $b_n > a_n$, then $W_{(n)}(s,t)=0$ for $a_n \le t \le b_n$ and
	$$\lp W_{(n)}(s,t) - W_{(n)}(s,b_n) \rp_{b_n \le t \le C_{(n)}(s)}$$
	is a rescaled Motzkin path.
\end{itemize}

In both cases,
$$\lp W_{(n)}(s',t) - W_{(n)}(s',a_n) \rp_{a_n \le t \le C_{(n)}(s')}$$
is also a rescaled Motzkin path---independent from $\lp W_{(n)}(s,t) - W_{(n)}(s,a_n) \rp_{a_n \le t \le C_{(n)}(s)}$.

\bigskip
Treating both cases separately, we obtain that there exists a constant $M$, independent of~$s$, such that for $n$ large enough,
$$\E{\sup_{a_n\le t\le C_{(n)}(s)} \lm W_{(n)}(s,t) - W_{(n)}(s,a_n) \rm^{q} \ \Big|\ C_{(n)}} \le M \, \lm C_{(n)}(s)-a_n \rm^{\frac q 2},$$
by Lemma \ref{rosenthal}. The same inequality holds with $s'$ instead of $s$. We have
\begin{align*}
\E{d_\K\lp W_{(n)}(s,\cdot),W_{(n)}(s',\cdot)\rp^{q} \ \Big|\ C_{(n)}}
	&\le M'\, \lp \ln C_{(n)} \rn_\alpha^{q} |s-s'|^{\alpha q} + \ln C_{(n)} \rn_\alpha^{\frac q 2} |s-s'|^{\alpha \frac q 2} \rp\\
	&\le M_q \lp \ln C_{(n)} \rn_\alpha^{q} \vee 1 \rp |s-s'|^{\alpha \frac q 2}.
\end{align*}
For $C \ge 1$,
\begin{equation}\label{kol}
\E{d_\K\lp W_{(n)}(s,\cdot),W_{(n)}(s',\cdot)\rp^{q} \ \Big|\ \ln C_{(n)}\rn_\alpha \le C} \le M_q \,C^{q}|s-s'|^{\alpha \frac q 2}.
\end{equation}

Let $0 <\delta < \frac 1 4$. Then, let $0 < \alpha < 1/2$ be such that $\delta < \alpha/2$, and $\eps>0$. Thanks to \eqref{kolfn}, we may find a constant $C$ such that, for $n$ sufficiently large,
$$\P\lp \ln C_{(n)} \rn_\alpha > C \rp < \eps.$$

For this $C$, the inequality \eqref{kol} allows us to apply Kolmogorov's criterion \cite[Theorem 3.3.16]{stroock99pta}: we find a constant $C'$ such that, for $n$ large enough,
$$\P\lp \sup_{s\neq s'} \frac {d_\K\lp W_{(n)}(s,\cdot) - W_{(n)}(s',\cdot) \rp}{\lm s-s'\rm^\delta} > C' \ \Big|\ \ln C_{(n)}\rn_\alpha \le C\rp < \eps.$$
Finally,
$$\P\lp \sup_{s\neq s'} \frac {d_\K\lp W_{(n)}(s,\cdot) - W_{(n)}(s',\cdot) \rp}{\lm s-s'\rm^\delta} > C' \rp < \frac \eps {1-\eps} + \eps,$$
which is what we needed.
\end{pre}

\begin{preuv}{of Proposition \ref{wsnake}}
We begin by showing the convergence of a finite number of trajectories, together with the whole contour process, and then conclude by a tightness argument using Lemma \ref{tight}.

\paragraph{Convergence of the finite-dimensional laws.}Let $p \ge 1$ and $0 \le s_1 < \dots < s_p < m$. We will show by induction on $p$ that
\begin{equation}\label{ind}
\lp (C_{(n)}(s))_{0 \le s\le m_{(n)}}, W_{(n)}(s_1,\cdot), \dots, W_{(n)}(s_p,\cdot) \rp \tol \lp F_{[0,m]}^{\sigma \to 0}, W(s_1,\cdot), \dots, W(s_p,\cdot) \rp.
\end{equation}
Because $m_{(n)} \to m$, for $n$ sufficiently large, $s_p \le m_{(n)}$ and the vector we consider is well-defined.

\bigskip
1) For $p=1$, we may only consider the case $s_1=0$. $(C_n(i))_{0\le i \le 2m_n+\sigma_n}$ is a discrete first-passage bridge on $[0,2m_n+\sigma_n]$ from $\sigma_n$ to $0$ and $W_n(0,j)=0$ for $0\le j \le \sigma_n$. Lemma \ref{fn} thus ensures us that
$$\lp (C_{(n)}(s))_{0\le s\le m_{(n)}}, (W_{(n)}(0,t))_{0\le t\le \sigma_{(n)}} \rp \tol \lp (F_{[0,m]}^{\sigma \to 0}(s))_{0 \le s\le m}, (W(0,t))_{0\le t\le \sigma} \rp.$$

\bigskip
2) Let us assume \eqref{ind} with $p-1$ instead of $p$. There exists a Motzkin path $M$, independent of $C_{(n)}$ and $W_{(n)}(s_i, \cdot)$, $1\le i \le p-1$, such that conditionally given 
$$\lp (C_{(n)}(s))_{0 \le s\le m_{(n)}}, W_{(n)}(s_1,\cdot), \dots, W_{(n)}(s_{p-1},\cdot) \rp,$$
for $0\le t \le C_{(n)}(s_p)$,
$$W_{(n)}(s_p,t) = W_{(n)}(s_{p-1},t\wedge a_n) + \frac{M_{\sqrt{2n}(t-a_n)^+}}{\g}$$
where $a_n \de \inf_{[s_{p-1}, s_p]}C_{(n)}$ and $x^+ \de x .\1{\{x\ge 0\}}$ stands for the positive part of $x$. The Donsker Invariance Principle \cite{billingsley68cpm} ensures that
$$\lp\frac {M_{\sqrt{2n} t}}{\g} \rp_{t\ge0}$$
converges weakly toward a Brownian motion $\beta$ for the uniform topology on every compact sets.

\bigskip
By means of the Skorokhod representation theorem (see e.g. \cite[Theorem 3.1.8]{ethier86mpc}), we may and will assume that this convergence holds almost surely. We also suppose that \eqref{ind} holds for~$p-1$.
Then, a.s.,
$$\lp W_{(n)}(s_p,t)\rp_{0\le t \le C_{(n)}(s_p)}
	\to \lp W(s_{p-1},t \wedge a) + \beta_{(t-a)^+}\rp_{0\le t \le F_{[0,m]}^{\sigma \to 0}(s_p)}$$
where $a \de \inf_{[s_{p-1}, s_p]}F_{[0,m]}^{\sigma \to 0}$. To see this, observe that
$$\lm C_{(n)}(s_p) - F_{[0,m]}^{\sigma \to 0}(s_p) \rm \to 0$$
and
\begin{align*}
\sup_t \lm W_{(n)}(s_{p-1},t\wedge a_n) - W(s_{p-1},t\wedge a) \rm
	&\le \sup_{0\le t \le a_n}\lm W_{(n)}(s_{p-1}, t) - W(s_{p-1}, t) \rm\\
		&\hspace{1cm}+\sup_{a_n \wedge a \le t \le a_n \vee a }\lm W(s_{p-1}, t) - W(s_{p-1}, a_n \wedge a) \rm\\
		& \to 0,
\end{align*}
by continuity of $W(s_{p-1},\cdot)$. A similar inequality holds for $M$.

\bigskip
Finally, the law of
$$\lp W(s_{p-1},t \wedge a) + \beta_{(t-a)^+}\rp_{0\le t \le F_{[0,m]}^{\sigma \to 0}(s_p)}$$
is that of $W(s_p,\cdot)$, conditionally given
$$\lp (F_{[0,m]}^{\sigma \to 0}(s))_{0 \le s\le m'}, W(s_1,\cdot), \dots, W(s_{p-1},\cdot) \rp,$$
which is precisely what we wanted.

\paragraph{Tightness.}Let $0 < \delta < 1/4$ and $\eps >0$. Lemma \ref{tight} provides us with a constant $C$ and an integer $n_0$ such that for all $n\ge n_0$, $\P(W_{(n)} \notin A) < \eps$, where
$$A\de \lb X \in \W_0\,:\, \sup_{s\neq s'} \frac {d_\K\lp X(s,\cdot) - X(s',\cdot) \rp}{\lm s-s'\rm^\delta} \le C \rb.$$

Let $(s_k)_{k\ge 1}$ be a countable dense subset of $[0,m)$. As for every $k\ge 1$, $\lp W_{(n)}(s_k,\cdot) \rp_n$ is tight, we can find compact sets $K_k \subseteq \W_0$ such that for all $k \ge 1$, for all $n \ge n_0$,
$$\P\lp W_{(n)}(s_k,\cdot) \notin K_k \rp < \frac \eps{2^k}.$$

The set
$$\mathbb K \de A \cap \lb X \in \W_0\,:\, \forall k \ge 1,\, X(s_k,\cdot) \in K_k \rb.$$
is a compact subset of $\W_0$ by Ascoli's theorem \cite[XX]{schwartz70atg} and for $n \ge n_0$, $\P\lp W_{(n)} \notin \mathbb K \rp <2\eps$, hence the tightness of the sequence of $W_{(n)}$'s laws.
\end{preuv}


\section{Proof of Theorem \ref{cvq}}\label{secpro}

We adapt the proof given in \cite{legall07tss} for the case $g=0$ to our case $g\ge 1$.

\subsection{Setting}

Let $\q_n$ be uniformly distributed over the set $\Q_n$ of bipartite quadrangulations of genus~$g$ with~$n$ faces. Conditionally given $\q_n$, we take $v_n$ uniformly over $V(\q_n)$ so that $(\q_n,v_n)$ is uniform over the set $\Qb_n$ of pointed bipartite quadrangulations of genus $g$ with $n$ faces. Recall that every element of $\Q_n$ has the same number of vertices: $n+2-2g$. Through the Chapuy-Marcus-Schaeffer bijection, $(\q_n,v_n)$ corresponds to a uniform well-labeled $g$-tree with $n$ edges $(\t_n,\l_n)$. The parameter $\eps \in \{-1,1\}$ appearing in the bijection will be irrelevant to what follows.

Recall the notations $\t_n(0)$, $\t_n(1)$, \dots, $\t_n(2n)$ and $\q_n(0)$, $\q_n(1)$, \dots, $\q_n(2n)$ from Section~\ref{seccms}. For technical reasons, it will be more convenient, when traveling along the $g$-tree, not to begin by its root but rather by the first edge of the first forest. Precisely, we define
$$\rr\t_n(i) \de \lb \begin{array}{cll}
			\t_n(i - u_n +2n) &\text{ if } &0 \le i \le u_n,\\
			\t_n(i - u_n) &\text{ if } &u_n \le i \le 2n,
		      \end{array}\rno$$
and
$$\rr\q_n(i) \de \lb \begin{array}{cll}
			\q_n(i - u_n +2n) &\text{ if } &0 \le i \le u_n,\\
			\q_n(i - u_n) &\text{ if } &u_n \le i \le 2n,
		      \end{array}\rno$$
where $u_n$ is the integer recording the position of the root in the first forest of $\t_n$. We endow $\ent 0 {2n}$ with the pseudo-metric $d_n$ defined by
$$d_n(i,j) \de d_{\q_n}\lp \rr\q_n(i),\rr\q_n(j)\rp.$$

We define the equivalence relation $\sim_n$ on $\ent 0 {2n}$ by declaring that $i \sim_n j$ if $\rr\q_n(i)=\rr\q_n(j)$, that is if $d_n(i,j) =0$. We call $\pi_n$ the canonical projection from $\ent 0 {2n}$ to $\ent 0 {2n}_{/\sim_n}$ and we slightly abuse notation by seeing $d_n$ as a metric on $\ent 0 {2n}_{/\sim_n}$ defined by $d_n(\pi_n(i),\pi_n(j)) \de d_n(i,j)$. In what follows, we will always make the same abuse with every pseudo-metric. The metric space $\lp \ent 0 {2n}_{/\sim_n},d_n \rp$ is then isometric to $\lp V(\q_n)\bs\{v_n\},d_{\q_n} \rp$, which is at $\dGH$-distance~$1$ from the space $\lp V(\q_n),d_{\q_n} \rp$.

\bigskip

We extend the definition of $d_n$ to non integer values by linear interpolation: for $s,t\in [0,2n]$,
\begin{equation}\label{extdn}
d_n(s,t)\de \ll s\, \ll t\, d_n(\lc s\rc,\lc t\rc) +
	      \ll s\, \cc t\, d_n(\lc s\rc,\lf t\rf) +
	      \cc s\, \ll t\, d_n(\lf s\rf,\lc t\rc) +
	      \cc s\, \cc t\, d_n(\lf s\rf,\lf t\rf),
\end{equation}
where $\lf s\rf \de \sup\{k \in \Z,\, k\le s\}$, $\lc s\rc \de \lf s \rf + 1$, $\ll s \de s - \lf s\rf$ and $\cc s \de \lc s\rc - s$. Beware that $d_n$ is no longer a pseudo-metric on $[0,2n]$: indeed, $d_n(s,s)=2\, \ll s\, \cc s\, d_n(\lc s\rc,\lf s\rf)>0$ as soon as $s \notin \Z$. The triangular inequality, however, remains valid for all $s,t\in [0,2n]$. Using the Chapuy-Marcus-Schaeffer bijection, it is easy to see that $d_n(\lc s\rc,\lf s\rf)$ is equal to either $1$ or $2$, so that $d_n(s,s) \le 1/2$.

\bigskip

As usual, we define the rescaled version: for $s,t \in [0,1]$, we let
\begin{equation}\label{rescdn}
d_{(n)}(s,t) \de \frac 1 {\g}\, d_n(2ns,2nt),
\end{equation}
so that
\begin{equation}\label{gh}
\dGH \lp \lp \frac 1 {2n} \ent 0 {2n}_{/\sim_n},d_{(n)} \rp, \lp V(\q_n),\frac 1 {\g}\, d_{\q_n} \rp \rp \le \frac 1 {\g}.
\end{equation}

\subsection{Tightness of the distance processes}

The first step is to show the tightness of the processes $d_{(n)}$'s laws. For that matter, we use the bound \eqref{lemmed}. We define
$$d_n^\circ(i,j) \de \l_n\lp \rr\t_n(i) \rp + \l_n\lp \rr\t_n(j)\rp - 2 \max \lp \min_{ k \in \overrightarrow{\ent i j}} \l_n\lp \rr\t_n(k)\rp ,\min_{ k \in \overrightarrow{\ent j i}} \l_n\lp \rr\t_n(k) \rp \rp +2,$$
we extend it to $[0,2n]$ as we did for $d_n$ by \eqref{extdn}, and we define its rescaled version $d_{(n)}^\circ$ as we did for $d_n$ by \eqref{rescdn}. We readily obtain the following bound,
\begin{equation}\label{bound2}
d_{(n)}(s,t) \le d^\circ_{(n)}(s,t).
\end{equation}

\subsubsection*{Expression of $d_{(n)}^\circ$ in terms of the spatial contour function of the $g$-tree}

Although it is not straightforward to define a contour function for the whole $g$-tree, we may define its spatial contour function $\L_n: [0,2n] \to \R$ by,
$$\L_n(i) \de \l_n\lp\rr\t_n(i)\rp - \l_n\lp\rr\t_n(0)\rp, \hspace{1cm}0\le i \le 2n,$$
and by linearly interpolating it between integer values. The rescaled version is then defined by
$$\L_{(n)} \de \lp \frac {\L_n(2nt)} {\g}\rp_{0\le t \le 1},$$
and we easily see that
$$d^\circ_{(n)}(s,t) = \L_{(n)}(s) + \L_{(n)}(t) - 2 \max \lp \min_{ x \in \overrightarrow{[s,t]}} \L_{(n)}(x),\min_{ x \in \overrightarrow{[t,s]}} \L_{(n)}(x) \rp + O\big( n^{\frac 14} \big)$$
where
$$\overrightarrow{[s,t]} \de \lb \begin{array}{cll}
					[s,t]			&\text{ if } & s \le t, \\
					\left[s,1\right] \cup [0,t]	&\text{ if } & t < s.
				    \end{array}\rno$$

\subsubsection*{Convergence results}

As in Section \ref{decomp}, we call $\s_n$ the scheme of $\t_n$, $(\f_n^\e,\l_n^\e)_{\e\in\vec E(\s_n)}$ its well-labeled forests, $(m_n^\e)_{\e\in\vec E(\s_n)}$ and $(\sigma_n^\e)_{\e\in\vec E(\s_n)}$ respectively their sizes and lengths, $(l_n^v)_{v\in V(\s_n)}$ the shifted labels of its nodes, $(\mM_n^{\e})_{\e \in \vec E(\s_n)}$ its Motzkin bridges, and $u_n$ the integer recording the position of the root in the first forest $\f_n^{\e_*}$.
We call $(C_n^\e,L_n^\e)$ the contour pair of the well-labeled forest $(\f_n^\e,\l_n^\e)$ and we extend the definition of $\mM_n^\e$ to $[0,\sigma_n^\e]$ by linear interpolation.

As usual, we define the rescaled versions of these objects
$$m_{(n)}^\e \de \frac {2m_n^\e + \sigma_n^\e}{2n}\scom
\sigma_{(n)}^\e \de \frac{\sigma_n^\e}{\sqrt{2n}}\scom
l_{(n)}^v \de \frac{l_n^v}{\g}\scom
u_{(n)} \de \frac{u_n}{2n}$$
and
$$C_{(n)}^\e \de \lp \frac {C_n^\e(2nt)} {\sqrt{2n}} \rp_{0\le t \le m_{(n)}^\e},\ 
L_{(n)}^\e \de \lp \frac {L_n^\e(2nt)} {\g}\rp_{0\le t \le m_{(n)}^\e},\ 
\mM_{(n)}^\e \de \lp \frac {\mM_n^\e(\sqrt {2n}\, t)}{\g} \rp_{0\le t \le \sigma_{(n)}^\e}.$$


Combining the results of Proposition \ref{cvint}, Lemma\footnote{Remark that $\g = \sqrt{\!\frac23}\, \sqrt{\!\!\sqrt{2n}}$.} \ref{bn} and Corollary \ref{snak}, we find that the vector
$$\lp \s_n, \big( m_{(n)}^\e\big)_{\e\in \vec E(\s_n)}, \big(\sigma_{(n)}^\e\big)_{\e\in \vec E(\s_n)}, \big( l_{(n)}^v\big)_{v\in V(\s_n)}, u_{(n)}, \big( C_{(n)}^\e, L_{(n)}^\e \big)_{\e\in \vec E(\s_n)}, \big( \mM_{(n)}^\e \big)_{\e\in \vec E(\s_n)} \rp$$
converges in law toward the random vector
$$\lp \s_\infty, \lp m_\infty^\e\rp_{\e\in \vec E(\s_\infty)}, \lp\sigma_\infty^\e\rp_{\e\in \vec E(\s_\infty)}, \lp l_\infty^v\rp_{v\in V(\s_\infty)}, u_\infty, \lp C_\infty^\e, L_\infty^\e \rp_{\e\in \vec E(\s_\infty)}, \lp \mM_\infty^\e \rp_{\e\in \vec E(\s_\infty)} \rp$$
whose law is defined as follows:
\begin{itemize}
	\item the law of the vector
	$$\mathfrak I_\infty \de \lp \s_\infty, \lp m_\infty^\e\rp_{\e\in \vec E(\s_\infty)}, \lp\sigma_\infty^\e\rp_{\e\in \vec E(\s_\infty)}, \lp l_\infty^v\rp_{v\in V(\s_\infty)}, u_\infty \rp$$
	is the probability $\mu$ defined before Proposition \ref{cvint},
	\item conditionally given $\mathfrak I_\infty$, 
	\begin{itemize}
		\item the processes $\lp C_\infty^\e, L_\infty^\e \rp$, ${\e\in \vec E(\s_\infty)}$ and $\lp \mM_\infty^\e \rp$, ${\e\in \ori E(\s_\infty)}$ are independent,
		\item the process $\lp C_\infty^\e, L_\infty^\e \rp$ has the law of a Brownian snake's head on $[0,m_\infty^\e]$ going from $\sigma_\infty^\e$ to $0$:
		$$\lp C_\infty^\e, L_\infty^\e \rp \law \lp F_{[0,m_\infty^\e]}^{\sigma_\infty^\e \to 0}, Z_{[0,m_\infty^\e]}\rp,$$
		\item the process $\lp \mM_\infty^\e \rp$ has the law of a Brownian bridge on $[0,\sigma_\infty^\e]$ from $0$ to $l_\infty^\e \de l_\infty^{\e^+} - l_\infty^{\e^-}$:
		$$\lp \mM_\infty^\e \rp \law B_{[0,\sigma_\infty^\e]}^{0 \to l_\infty^\e},$$
		\item the Motzkin bridges are linked through the relation
		$$\mM_\infty^{\bar\e}(s) = \mM_\infty^\e(\sigma_\infty^\e-s) - l_\infty^\e.$$
	\end{itemize}
\end{itemize}

Applying the Skorokhod theorem, we may and will assume that this convergence holds almost surely. As a result, note that for $n$ large enough, $\s_n=\s_\infty$.

\subsubsection*{Decomposition of $\L_{(n)}$ along the forests}

In order to study the convergence of $\L_{(n)}$, we will express it in terms of the $L_{(n)}^\e$'s and $\mM_{(n)}^\e$'s. First, the labels in the forest $(\f_n^\e,\l_n^\e)$ are to be shifted by the value of the Motzkin path $\mM_n^\e$ at the time telling which subtree is visited: recall the definition \eqref{lele} of the process
$$\L_n^\e \de \Big( L_n^\e(t) + \mM_n^\e\big( \sigma_n^\e - \underline C_n^\e(t) \big) \Big)_{0\le t \le 2m_n^\e +\sigma_n^\e}.$$
We define its rescaled version
$$\L_{(n)}^\e \de \lp \frac {\L_n^\e(2nt)} {\g}\rp_{0\le t \le m_{(n)}^\e} = \lp L_{(n)}^\e(t) + \mM_{(n)}^\e\big( \sigma_{(n)}^\e - \underline C_{(n)}^\e(t) \big) \rp_{0\le t \le m_{(n)}^\e},$$
as well as its limit in the space $\lp \K, \, d_\K \rp$,
$$\L_{(n)}^\e \ton \L_\infty^\e \de \Big( L_\infty^\e(t) + \mM_\infty^\e\big( \sigma_\infty^\e - \underline C_\infty^\e(t) \big) \Big)_{0\le t \le m_\infty^\e}.$$

\bigskip

We then need to concatenate these processes. For $f,g\in \K_0$ two functions started at $0$, we call $f \bullet g \in \K_0$ their concatenation defined by $\sigma(f \bullet g) \de \sigma(f) + \sigma(g)$ and, for $0\le t \le \sigma(f \bullet g)$,
$$f \bullet g(t) \de \lb \begin{array}{cll}
					f(t)				& \text{ if } & 0\le t \le \sigma(f),\\
					f(\sigma(f)) + g(t-\sigma(f))	& \text{ if } & \sigma(f) \le t \le \sigma(f) + \sigma(g).
				     \end{array}\rno$$

We sort the half-edges of $\s_n$ according to its facial order, beginning with the root: $\e_1=\e_*$, \dots, $\e_\k$ and we see that
$$\L_{(n)} = \L_{(n)}^{\e_1} \bullet \L_{(n)}^{\e_2} \bullet \dots \bullet \L_{(n)}^{\e_\k}.$$
We also sort the half-edges of $\s_\infty$ in the same way and define $\L_\infty\de\L_\infty^{\e_1} \bullet \L_\infty^{\e_2} \bullet \dots \bullet \L_\infty^{\e_\k}$.

\begin{lem}
The concatenation is continuous from $(\K_0,d_\K)^2$ to $(\K_0,d_\K)$.
\end{lem}

\begin{proof}
Let $(f_n,g_n)$ be a sequence of functions in $\K_0^2$ converging toward $(f,g) \in \K_0^2$ and $\eps >0$. There exist an $0 <\eta < \eps$ and an $n_0$ such that
$$|s - t| < \eta \Rightarrow |f\bullet g(s) - f\bullet g(t)|  < \eps \sand n \ge n_0 \Rightarrow d_\K(f_n,f)\, \vee\, d_\K(g_n,g) <\eta.$$

Let $0 \le t \le \sigma(f\bullet g) \wedge \sigma(f_n \bullet g_n)$ and $n \ge n_0$ be fixed. If $t \le \sigma(f_n)$, we call $\tilde t \de t\wedge \sigma(f)$. In that case,
$$|f_n \bullet g_n(t) - f\bullet g(\tilde t)| = |f_n (t) - f(t \wedge \sigma(f))| \le d_\K(f_n,f) < \eps.$$
If $\sigma(f_n) < t $, we call $\tilde t \de \sigma(f) + (t-\sigma(f_n))\wedge \sigma(g)$ and we have
$$|f_n \bullet g_n(t) - f\bullet g(\tilde t)| = |g_n ((t-\sigma(f_n))\wedge \sigma(g_n)) - g((t-\sigma(f_n))\wedge \sigma(g))| \le d_\K(g_n,g) < \eps.$$
In both cases, $|t-\tilde t| < \eta$, so that $|f\bullet g(\tilde t) - f\bullet g(t)|  < \eps$. Hence
\qed86$$d_\K(f_n \bullet g_n,f \bullet g) < |\sigma(f_n) - \sigma(f)| + |\sigma(g_n) - \sigma(g)| + 2 \eps < 4\eps.$$
\end{proof}

This ensures us that $\L_{(n)}$ converges in $(\K,d_\K)$ toward $\L_\infty$, so that $\lp d^\circ_{(n)}(s,t)\rp_{0 \le s,t \le 1}$ converges in $\lp \C([0,1]^2,\R),\|\cdot\|_\infty \rp$ toward $\lp d^\circ_\infty(s,t)\rp_{0 \le s,t \le 1}$ defined by
$$d^\circ_\infty(s,t) \de \L_\infty(s) + \L_\infty(t) - 2 \max \lp \min_{ x \in \overrightarrow{[s,t]}} \L_\infty(x),\min_{ x \in \overrightarrow{[t,s]}} \L_\infty(x) \rp.$$

\subsubsection*{Tightness}

\begin{lem}\label{tightd}
The sequence of the laws of the processes
$$\lp d_{(n)}(s,t) \rp_{0 \le s,t \le 1}$$
is tight in the space of probability measure on $\C([0,1]^2,\R)$.
\end{lem}

\begin{pre}
First observe that, for every $s$, $s'$, $t$, $t'\in[0,1]$,
$$\lm d_{(n)}(s,t) - d_{(n)}(s',t') \rm \le d_{(n)}(s,s') + d_{(n)}(t,t') \le d^\circ_{(n)}(s,s') + d^\circ_{(n)}(t,t').$$

By Fatou's lemma, we have for every $k \in \N$ and $\delta > 0$,
$$\limsup_{n\to \infty} \P \lp \sup_{|s-s'| \le \delta} d_{(n)}^\circ(s,s') \ge 2^{-k} \rp \le \P \lp \sup_{|s-s'| \le \delta} d_\infty^\circ(s,s') \ge 2^{-k} \rp.$$
Since $d_\infty^\circ$ is continuous and null on the diagonal, for $\eps >0$, we may find $\delta_k > 0$ such that, for~$n$ sufficiently large,
\begin{equation}\label{fat}
\P \lp \sup_{|s-s'| \le \delta_k} d_{(n)}^\circ(s,s') \ge 2^{-k} \rp \le 2^{-k}\eps.
\end{equation}

By taking $\delta_k$ even smaller if necessary, we may assume that the inequality \eqref{fat} holds for all $n\ge 1$. Summing over $k\in \N$, we find that for every $n\ge 1$,
$$\P\lp d_{(n)} \in \K_\eps \rp \ge 1- \eps,$$
where
$$\K_\eps \de \bigg\{ f \in \C([0,1]^2,\R )\,:\, f(0,0)=0,\, \forall k \in \N,\, \sup_{|s-s'|\wedge |t-t'| \le \delta_k} \lm f(s,t)-f(s',t') \rm \le 2^{1-k} \bigg\}$$
is a compact set.
\end{pre}

\subsection{The genus $g$ Brownian map}

\subsubsection*{Proof of the first assertion of Theorem \ref{cvq}}

Thanks to Lemma \ref{tightd}, there exist a subsequence $(n_k)_{k\ge 0}$ and a function $d_\infty \in \C([0,1]^2,\R)$ such that
\begin{equation}\label{dinfty}
\lp d_{(n_k)}(s,t) \rp_{0 \le s,t \le 1} \tode{(d)}{k}{\infty} \lp d_\infty(s,t) \rp_{0 \le s,t \le 1}.
\end{equation}
By the Skorokhod theorem, we will assume that this convergence holds almost surely. As the $d_{(n)}$ functions, the function $d_\infty$ obeys the triangular inequality. And because $d_{(n)}(s,s) = O(n^{-1/4})$ for all $s\in [0,1]$, the function $d_\infty$ is actually a pseudo-metric. We define the equivalence relation associated with it by saying that $s \sim_\infty t$ if $d_\infty(s,t)=0$, and we call $\q_\infty \de [0,1]_{/\sim_\infty}$.

We will show the convergence claimed in Theorem \ref{cvq} along the same subsequence $(n_k)_{k\ge 0}$. Thanks to \eqref{gh}, we only need to see that
$$\dGH \lp \lp (2n_k)^{-1} \ent 0 {2n_k}_{/\sim_{n_k}},d_{(n_k)} \rp, \lp \q_\infty,d_\infty \rp \rp \tok 0.$$

For that matter, we will use the characterization of the Gromov-Hausdorff distance via correspondences. Recall that a correspondence between two metric spaces $(\S,\delta)$ and $(\S',\delta')$ is a subset $\rR\subseteq \S \times \S'$ such that for all $x\in \S$, there is at least one $x'\in \S'$ for which $(x,x')\in \rR$ and vice versa. The distortion of the correspondence $\rR$ is defined by
$$\dis(\rR) \de \sup \lb |\delta(x,y) - \delta(x',y')|\,:\, (x,x'),(y,y')\in \rR \rb.$$

Then we have \cite[Theorem 7.3.25]{burago01cmg}
$$\dgh(\S,\S') = \frac 12 \inf_{\rR} \dis(\rR)$$
where the infimum is taken over all correspondences between $\S$ and $\S'$.

\bigskip

We define the correspondence $\rR_n$ between $\lp (2n)^{-1} \ent 0 {2n}_{/\sim_n},d_{(n)} \rp$ and $\lp \q_\infty,d_\infty \rp$ as the set 
$$\rR_n \de \lb \lp (2n)^{-1}\, \pi_{n}(\lf 2n t \rf),\pi_\infty(t)\rp, \, t\in [0,1] \rb$$
where $\pi_n : \ent 0 {2n} \to \ent 0 {2n}_{/\sim_n}$ and $\pi_\infty : [0,1] \to \q_\infty$ are both canonical projections. Its distortion is
$$\dis(\rR_n) 
	    = \sup_{0 \le s,t \le 1} \Big| d_{(n)}\lp \frac{\lf 2n s \rf}{2n},\frac{\lf 2n t \rf}{2n}\rp - d_\infty(s,t) \Big|,$$
and, thanks to \eqref{dinfty},
$$\dGH \lp \lp (2n_k)^{-1} \ent 0 {2n_k}_{/\sim_{n_k}},d_{(n_k)} \rp, \lp \q_\infty,d_\infty \rp \rp \le \frac 12\, \dis \lp \rR_{n_k} \rp \tok 0.$$

\subsubsection*{A bound on $d_\infty$}

If we take the limit of the inequality \eqref{bound2} along the subsequence $(n_k)_{k\ge 0}$, we find $d_\infty(s,t) \le d^\circ_\infty(s,t)$. Because $d_\infty^\circ$ does not satisfy the triangular inequality, we may improve this bound by considering the largest metric on $\q_\infty$ that is smaller than $d_\infty^\circ$: for all $a$ and $b\in \q_\infty$, we have
$$d_\infty(a,b) \le d_\infty^*(a,b) \de \inf  \lb \sum_{i=0}^k d_\infty^\circ(s_i,t_i) \rb$$
where the infimum is taken over all integer $k\ge0$ and all sequences $s_0$, $t_0$, $s_1$, $t_1$,\dots, $s_k$, $t_k$ satisfying $a = \pi_\infty(s_0)$, for all $0 \le i \le k-1$, $t_i \sim_\infty s_{i+1}$, and $b= \pi_\infty(t_k)$.

\subsection{Hausdorff dimension of the genus $g$ Brownian map}

We now prove the second assertion of Theorem \ref{cvq}. We follow the method provided by Le Gall and Miermont \cite{legall09scaling}. As usual, we proceed in two steps.

\subsubsection*{Upper bound}

Let $0 < \alpha < \frac 14$. For every $\e\in\vec E(\s_\infty)$, Lemmas \ref{lemb} and \ref{tight}, together with \eqref{kolfn}, imply that $\L_\infty^\e$ is $\alpha$-H\"older. The same goes for $\L_\infty$ by finite concatenation. This yields that the canonical projection $\pi_\infty : ([0,1],|\cdot|) \to (\q_\infty,d_\infty)$ is $\alpha$-H\"older as well: for $0\le s,t \le 1$,
$$d_\infty(\pi_\infty(s),\pi_\infty (t)) = d_\infty(s,t) \le d^\circ_\infty(s,t) \le 2 \|\L_\infty\|_\alpha \, |s-t|^\alpha.$$
It follows that $\dH(\q_\infty,d_\infty) \le \frac 1 \alpha \dH([0,1])$. Taking the infimum over $\alpha \in (0,1/4)$, we have
$$\dH(\q_\infty,d_\infty) \le 4.$$

\subsubsection*{Lower bound}

We start with a lemma giving a lower bound on $d_\infty(s,t)$. Let us first define a contour function $\CC_n: [0, 2n] \to \R_+$ for the $g$-tree $\t_n$ by
$$\CC_n \de \big( C_n^{\e_1} - \sigma_n^{\e_1} \big) \bullet \big( C_n^{\e_1} - \sigma_n^{\e_2} \big) \bullet \dots \bullet \big( C_n^{\e_1} - \sigma_n^{\e_\k} \big) + \sum_{i=1}^{\k} \sigma_n^{\e_i}$$
where the half-edges $\e_1=\e_*$, \dots, $\e_\k$ are sorted according to the facial order of $\s_n$. This function is actually the contour function of the ``large'' forest consisting in the concatenation of $\f_n^{\e_1}$, $\f_n^{\e_2}$, \dots, $\f_n^{\e_\k}$. As usual, we define its rescaled version $\CC_{(n)}$, as well as its limit
$$\CC_{(n)} \ton \CC_\infty \de \big( C_\infty^{\e_1} - \sigma_\infty^{\e_1} \big) \bullet \big( C_\infty^{\e_1} - \sigma_\infty^{\e_2} \big) \bullet \dots \bullet \big( C_\infty^{\e_1} - \sigma_\infty^{\e_\k} \big) + \sum_{i=1}^{\k} \sigma_\infty^{\e_i}$$
where, this time, the half-edges are sorted according to the facial order of $\s_\infty$.

\bigskip

For $0 \le s, t \le 1$, we define the set
$$\lin_\infty(s,t) \de \lb s\wedge t \le x \le s\vee t\,:\, \underline \CC_\infty(x)= \underline \CC_\infty(s),\, \CC_\infty(x) = \inf_{[x\wedge s,\, x\vee s]} \CC_\infty \rb.$$
It will become clearer in a moment what this set represents, while looking at its discrete analog.

\begin{lem}\label{bounddinf}
The following bound holds
$$d_\infty(s,t) \ge \L_\infty(s) - \min_{\lin_\infty(s,t)} \L_\infty$$
\end{lem}

\begin{proof}
This inequality follows easily by approximation, once we have shown its discrete analog:
\begin{equation}\label{lbound}
d_n(i,j) \ge \L_n(i) - \min_{\lin_n(i,j)} \L_n
\end{equation}
where the set
$$\lin_n(i,j) \de \lb i\wedge j \le k \le i\vee j\,:\, \underline \CC_n(k)= \underline \CC_n(i),\, \CC_n(k) = \inf_{[k\wedge i,\, k\vee i]} \CC_n \rb$$
represents the ancestral lineage of $\rr\t_n(i)$ between $i$ and $j$. An integer $k$ belongs to $\lin_n(i,j)$ \iff $k$ is between $i$ and $j$ (first constraint), $\rr\t_n(k)$ lies in the same subtree as $\rr\t_n(i)$ (second constraint), and $\rr\t_n(k)$ is an ancestor of $\rr\t_n(i)$ (third constraint). Beware that $\lin_n(j,i)$ is in general a totally different set.

We can suppose $i \neq j$. In order to show \eqref{lbound}, we consider a geodesic path $\gamma_0$, $\gamma_1$, \dots, $\gamma_{d_n(i,j)}$ from $\rr\t_n(i)$ to $\rr\t_n(j)$ and call $k\in \lin_n(i,j)$ an integer for which $\L_n(k) = \min_{\lin_n(i,j)} \L_n$. Let us call~$p$ the order of the vertex $\rr\t_n(k)$. Then removing the edges incident to $\rr\t_n(k)$ breaks $\t_n$ into $p+1$ connected components: $\{\rr\t_n(k)\}$, $p-1$ trees, and a $(p+1)$-th component (which is a $g$-tree, unless if $\rr\t_n(k)$ belongs to the floor of a forest). One of these components contains $\rr\t_n(i)$ and another one contains $\rr\t_n(j)$. Say that $\gamma_r$, $r < d_n(i,j)$ is the last vertex of the geodesic path lying in the same component as $\rr\t_n(i)$. Then $\gamma_r$ is linked by an edge of $\q_n$ to $\gamma_{r+1}$, which lies in another component. Moreover, the facial sequence of $\t_n$ must visit $\rr\t_n(k)$ between any time it visits $\gamma_r$ and any time it visits $\gamma_{r+1}$ (in that order or the other). The way we construct edges in the Chapuy-Marcus-Schaeffer bijection thus imposes $\l_n(\rr\t_n(k)) \ge \l_n(\gamma_r) \vee \l_n(\gamma_{r+1})$. Finally,
$$d_n(i,j) \ge d_{\q_n}(\rr\q_n(i),\gamma_r) \ge d_{\q_n}(\rr\q_n(i),v_n) - d_{\q_n}(v_n,\gamma_r) = \l_n(\rr\t_n(i)) - \l_n(\gamma_r),$$
and the same holds with $r+1$ instead of $r$, yielding
\qed{11}{10}$$d_n(i,j) \ge \l_n(\rr\t_n(i)) - \l_n(\rr\t_n(k)) = \L_n(i) - \min_{\lin_n(i,j)} \L_n.$$
\end{proof}

Let us define the measure $\lambda$ on $\q_\infty$ as the image of the Lebesgue measure on $[0,1]$ by the canonical projection $\pi_\infty : [0,1] \to \q_\infty$. From now on, we work conditionally given the parameters vector $\mathfrak I_\infty$. Let $0 \le s \le 1$ be a point that is not of the form $\sum_{i=1}^k m_\infty^{\e_i}$ for some $k=0$, \dots, $\k$. This means that it is not $0$, $1$, or a point at which two functions are being concatenated. Such points will thereafter be called \textit{junction points}.

Suppose that for some $\delta > 0$, we can find two positive numbers $r_-$ and $r_+$ such that
\begin{equation}\label{rplusmoins}
\L_\infty(s) - \min_{\lin_\infty(s,s-r_-)} \L_\infty > \delta \sand \L_\infty(s) - \min_{\lin_\infty(s,s+r_+)} \L_\infty > \delta.
\end{equation}
For $a \in \q_\infty$ and $r>0$, we call $B_\infty(a,r)$ the open ball centered at $a$ with radius $r$ for the metric $d_\infty$. Using Lemma \ref{bounddinf} and the elementary fact that $\lin_\infty(s,t) \subseteq \lin_\infty(s,t')$ as soon as $|t-s| \le |t'-s|$, we find that $B_\infty(\pi_\infty(s),\delta) \subseteq \pi_\infty\big((s-r_-, s+ r_+)\big)$. As a result, we would have $\lambda(B_\infty(\pi_\infty(s),\delta)) \le r_- +r_+$.

\bigskip

For all $0 \le x \le \CC_\infty(s)-\underline \CC_\infty(s)$, we define
$$\tau_x \de \inf\lb r \ge s,\, \CC_\infty(r)=\CC_\infty(s)-x \rb$$
and we see that $\lin_\infty(s,\tau_x) = \{ \tau_y,\,0\le y \le x\}$. The discussion preceding Section \ref{snhe} shows that the process
$$\Big(\L_\infty(\tau_x) - \L_\infty(s)\Big)_{0 \le x \le \CC_\infty(s)-\underline \CC_\infty(s)}$$
has the law of a real Brownian motion started from $0$. Let $\eta >0$. Almost surely, provided that $\CC_\infty(s)-\underline \CC_\infty(s) >0$, the law of the iterated logarithm ensures us that for $x$ small enough,
$$\inf_{0\le y \le x} (\L_\infty(\tau_y) - \L_\infty(s)) < - x^{\frac 12 + \eta},$$
so that
$$\L_\infty(s) - \min_{\lin_\infty(s,\tau_x)} \L_\infty = \L_\infty(s) - \inf_{0\le y\le x} \L_\infty(\tau_y) > x^{\frac 12 + \eta}.$$

We choose $\delta = x^{\frac 12 + \eta}$ and $r_+ = \tau_x - s$ so that the second part of \eqref{rplusmoins} holds. Moreover, because $s$ is not a junction point, on one of its neighborhoods, the function $\CC_\infty$ is a first-passage Brownian bridge, and is then absolutely continuous with respect to the Wiener measure on this neighborhood. It therefore obeys the law of the iterated logarithm as well. So, a.s., for~$r$ small enough,
$$\inf_{0\le t \le r} (\CC_\infty(s+t) - \CC_\infty(s)) < - r^{\frac 12 + \eta}.$$
It follows that $r_+ \le x^{(\frac 12 + \eta)^{-1}} = \delta^{(\frac 12 + \eta)^{-2}} = \delta^{4-\eta'}$ for some $\eta' > 0$. In a similar way, we can find an~$r_- < \delta^{4-\eta'}$ satisfying the first part of \eqref{rplusmoins}. This yields, for all $\delta >0$ small enough,
$$\lambda(B_\infty(\pi_\infty(s),\delta)) \le 2 \delta^{4-\eta'},$$
which implies that, for all $\eta'> 0$,
\begin{equation}\label{densthm}
\limsup_{\delta \to 0} \frac{\lambda(B_\infty(\pi_\infty(s),\delta))}{\delta^{4-\eta'}} \le 2.
\end{equation}

Once again, because $\CC_\infty$ is absolutely continuous with respect to the Wiener measure on a neighborhood of $s$, a.s. $\CC_\infty(s)-\underline \CC_\infty(s) >0$. For the record, note that if $s$ was a junction point, we would always have $\CC_\infty(s)=\underline \CC_\infty(s)$ by definition of a first-passage bridge. We obtain that for every $s$ that is not a junction point, \eqref{densthm} holds almost surely. Finally, as there are only $\k + 1$ junction points, Fubini-Tonelli's theorem shows that a.s., for $\lambda$-almost every~$a$,
$$\limsup_{\delta \to 0} \frac{\lambda(B_\infty(a,\delta))}{\delta^{4-\eta'}} \le 2.$$
We then conclude that $\dH(\q_\infty,d_\infty) \ge 4 - \eta'$ for all $\eta' >0$ by standard density theorems for Hausdorff measures (\cite[Theorem 2.10.19]{federer69gmt}).


\section{An expression of the constant $t_g$}

This section is dedicated to the proof of Theorem \ref{tgprop}. Recall that the constant $t_g$ is defined by: $|\Q_n|\sim t_g\, n^{\frac 52 (g-1)}\,12^n$. The relation \eqref{qt} gives that $|\T_n|\sim \frac12\, t_g\, n^{\frac {5g-3}2}\,12^n$, so that, thanks to \eqref{tg2},
$$t_g = 2^{\frac{3g+1}2}\, 3^g \ \Upsilon$$
where $\Upsilon$ was defined by \eqref{upsi}. For a given $\s\in\Sg^*$, we will concentrate on
\begin{equation}\label{intpart}
\int_{\mathcal{S}^\s} d\mathcal{L}^\s \ \1{\lb m^{\e_*}\ge 0,\,u< m^{\e_*}\rb}\ 
		\prod_{\e\in{\vec E}(\s)} - \ga'_{m^\e} \lp \sigma^\e \rp
		\prod_{\e\in\ori{E}(\s)} \ga_{\sigma^\e} \lp l^\e \rp.
\end{equation}

First, notice that by integrating with respect to $u$, only a factor $m^{\e_*}$ appears.

\subsection{Integrating with respect to $(m^\e)_{\e\in \vec E(\s)\bs \{\e_*\}}$}

For $\e \neq \e_*$, $m^\e$ is only present in the factor
\begin{equation}\label{factor}
m^{\e_*} \lp -\ga'_{m^{\e_*}}(\sigma^{\e_*})\rp \lp -\ga'_{m^\e}(\sigma^\e) \rp =
	\sigma^{\e_*} \, \ga_{m^{\e_*}}(\sigma^{\e_*}) \lp -\ga'_{m^\e}(\sigma^\e) \rp,
\end{equation}
so we have to compute an integral of the form given in the following lemma:

\begin{lem}\label{lemag}
Let $a$, $b$, and $t$ be three positive numbers. Then
$$\int_0^t \ga_{t-m}(a) \lp - \ga'_m(b) \rp dm= \ga_t(a+b).$$
\end{lem}

\begin{pre}
Let us call $f_t(a,b)$ the integral we have to compute, that is
$$f_t(a,b) = \frac b {2\pi} \int_0^t (t-m)^{-\frac12} \, m^{-\frac32} \, e^{-\frac12 \lp \frac{a^2}{t-m} +\frac{b^2}m \rp} dm.$$
By doing the change of variable $m \mapsto \frac m {t-m}$, we find
$$f_t(a,b) = \frac b {2\pi t} \int_0^\infty x^{-\frac32} \, e^{-\frac1{2t} \lp a^2(1+ x) + b^2\lp 1+ \frac{1}x\rp \rp}  dx$$

The change of variable $x \mapsto \frac{a^2}{b^2} \, x$ in this integral yields the identity
\begin{equation}\label{abba}
f_t(a,b) = f_t(b,a).
\end{equation}

When differentiating with respect to $a$, a factor $-\frac a t (1+x)$ appears inside the integral. We may split it into two terms, the first one being merely $-\frac a t f_t(a,b)$ and the second one being
$$-\frac b t \frac a {2\pi t} \int_0^\infty x^{-\frac12} \, e^{-\frac1{2t} \lp a^2(1+ x) + b^2\lp 1+ \frac{1}x\rp \rp}  dx=-\frac b t f_t(b,a) = -\frac b t f_t(a,b),$$
thanks to the change of variable $x\mapsto \frac 1 x$. All in all, we obtain
$$\partial_a f_t(a,b) = - \frac {a+b} t f_t(a,b),$$
so that there exists a function $g_t$ satisfying
$$f_t(a,b) = e^{-\frac1{2t} \lp a + b \rp^2}\, g_t(b).$$

Because of \eqref{abba}, the function $g_t$ is actually constant and
$$g_t(b)=e^{\frac{1}{2t}}\,f_t(0,1) = \frac 1 {2\pi t} \int_0^\infty x^{-\frac32} \, e^{-\frac1{2tx}}  dx
	 = \frac 1 {2\pi t} \int_0^\infty e^{-\frac1{2t} y^2} dy = \frac 1 {\sqrt{2\pi t}}.$$
	 
Putting all this together, we obtain the result.
\end{pre}

The first time we integrate with respect to an $m^\e$, for an $\e \neq \e_*$, we apply Lemma \ref{lemag} with $a=\sigma^{\e_*}$, $b=\sigma^\e$ and $t=m^{\e_*} + m^\e$ ($t$ does not depend on $m^\e$) and the factor \eqref{factor} is changed into
$$\sigma^{\e_*} \, \ga_{m^{\e_*}+ m^\e}(\sigma^{\e_*}+ \sigma^\e).$$

We may then apply Lemma \ref{lemag} again, with $a=\sigma^{\e_*} + \sigma^{\e}$, $b=\sigma^{\e'}$ and $t=m^{\e_*} + m^\e + m^{\e'}$ when integrating with respect to $m_{\e'}$ and so on. In the end, after integrating with respect to $u$ and $(m^\e)_{\e \neq \e_*}$, the
$$\1{\lb m^{\e_*}\ge 0,\,u< m^{\e_*}\rb}\ \prod_{\e\in{\vec E}(\s)} - \ga'_{m^\e} \lp \sigma^\e \rp$$
part in the integrand of \eqref{intpart} merely becomes
$$\sigma^{\e_*} \, \ga_1 \lp \textstyle\sum_{\e\in\vec E(\s)} \sigma^{\e}\rp = 
	\sigma^{\e_*} \, \ga \lp 2\textstyle\sum_{\e\in\ori E(\s)} \sigma^{\e}\rp =
	\frac{\sigma^{\e_*}}{\sqrt{2\pi}}\, e^{- 2\lp \sum_{\e\in\ori{E}(\s)}\sigma^{\e} \rp^2}.$$

\subsection{Integrating with respect to $(\sigma^\e)_{\e\in \ori E(\s)}$}

We call $s=\sum_{\e\in\ori E(\s)} \sigma^\e$. In order to integrate with respect to $(\sigma^\e)_{\e\in \ori E(\s)}$, we will integrate with respect to $s$ and with respect to $(\sigma^\e)_{\e\in \ori E(\s)}$ on the simplex, precisely,
$$d(\sigma^\e)_{\e\in \ori E(\s)}=ds\, \1{\lb\sigma^{\e_*} >0\rb} \,d(\sigma^\e)_{\e\in \ori E(\s)\bs\{\e_*\}},$$
where $\sigma^{\e_*}=s-\sum_{\e\in\ori E(\s)\bs\{\e_*\}} \sigma^\e$.

We then do the changes of variables $\sigma^\e \mapsto s\, \sigma^\e$ and $l^v \mapsto \sqrt{s} \,l^v$ for all $\e\neq\e_*$ and $v\neq \e_*^-$, so that $\sigma^{\e_*}$ becomes $1-\sum_{\e\in\ori E(\s)\bs\{\e*\}} \sigma^\e$ and the integral \eqref{intpart} becomes
$$\int d(l^v) \, \frac 1 {\sqrt{2\pi}} \int_0^\infty ds\, s^{5g-3} \, e^{-2s^2}
		\int d(\sigma^\e)_{\e\neq\e_*}  \1{\lb\sigma^{\e_*} >0\rb} \ \sigma^{\e_*} \prod_{\e\in\ori{E}(\s)} \ga_{\sigma^\e} \lp l^\e \rp.$$

The first part is easily enough dealt with,
$$\int_0^\infty ds\, s^{5g-3} \, e^{-2s^2} = 2^{-\frac{5g}2}\, \Gamma \lp \frac{5g}2 -1 \rp.$$
We then focus on
$$\int d(\sigma^\e)_{\e\neq\e_*}  \1{\lb\sigma^{\e_*} >0\rb} \ \sigma^{\e_*} \prod_{\e\in\ori{E}(\s)} \ga_{\sigma^\e} \lp l^\e \rp = \varphi\lp (|l^\e|)_{\e\in \ori{E}(\s)} \rp,$$
where the function $\varphi$ is defined, for $x^\e >0$, $\e \in \ori{E}(\s)$ by
$$\varphi\lp (x^\e)_{\e\in \ori{E}(\s)} \rp \de \int d(\sigma^\e)_{\e\neq\e_*}  \1{\lb\sigma^{\e_*} >0\rb} \ \sigma^{\e_*} \prod_{\e\in\ori{E}(\s)} \ga_{\sigma^\e} \lp x^\e \rp.$$

If we differentiate this function $\varphi$ with respect to every variables $x^\e$, we recognize the same integral we treated while integrating with respect to $(m^\e)$,
\begin{align*}
\prod_{\e\in \ori{E}(\s)}\hspace{-1mm} \lp -\partial_{x^\e}\rp \, \varphi\lp (x^\e)_{\e\in \ori{E}(\s)} \rp 
							& = \int d(\sigma^\e)_{\e\neq\e_*}  \1{\lb\sigma^{\e_*} >0\rb} \ \sigma^{\e_*} \prod_{\e\in\ori{E}(\s)} \lp -\ga'_{\sigma^\e} \lp x^\e \rp \rp\\
							& = x^{\e_*}\, \ga_1 \lp \textstyle\sum_{\e\in \ori{E}(\s)} x^\e \rp.
\end{align*}
Integrating back, we obtain
$$\varphi\lp (x^\e)_{\e\in \ori{E}(\s)} \rp = \ga^{[6g-1]}\lp \textstyle\sum_{\e\in \ori{E}(\s)} x^\e \rp + x^{\e_*}\, \ga^{[6g-2]}\lp \textstyle\sum_{\e\in \ori{E}(\s)} x^\e \rp,$$
where, for all $n\ge 1$, the functions $\ga^{[n]}$ are defined by
\begin{equation}\label{pn}
\ga^{[n]}(y) \de \int_y^\infty \hspace{-3mm} dy_{n-1} \int_{y_{n-1}}^\infty \hspace{-3mm} dy_{n-2} \dots \int_{y_2}^\infty \hspace{-3mm} dy_1 \ \ga_1(y_1).
\end{equation}

The integral \eqref{intpart} is now equal to some constant times
\begin{equation}\label{intdl}
\int d(l^v)_{v\neq \e_*^-} \ \ga^{[6g-1]}\lp \textstyle\sum_{\e\in \ori{E}(\s)} |l^\e| \rp + |l^{\e_*}|\, \ga^{[6g-2]}\lp \textstyle\sum_{\e\in \ori{E}(\s)} |l^\e| \rp.
\end{equation}

\subsection{Integrating with respect to $(l^v)_{v\in V(\s)\bs\{\e_*^-\}}$}

We follow here the ideas of \cite{chapuy07brm}. The term $\sum_{\e\in \ori{E}(\s)} |l^\e|$ is a linear combination of $l^v$'s. We will break the integral \eqref{intdl} into parts on which these coefficients are constant. This happens when the vertex labels are sorted according to a given ordering: we call $\O_\s$ the set of bijections from $\ent 0 {4g-3}$ into $V(\s)$.

Let $\lambda\in \O_\s$ be an ordering and $v\in V(\s)$. Because $\s$ is dominant, $v$ is connected to exactly three other vertices---not necessarily distinct---that we call $v'\!$, $v''\!$, and $v'''\!$. When the labels are sorted according to $\lambda$, that is when $l^{\lambda_0} < l^{\lambda_1} < \dots < l^{\lambda_{4g-3}}$, the coefficient of $l^v$ in the sum $\sum_{\e\in \ori{E}(\s)} |l^\e|$ is
$$c(\lambda,v) \de 2\lp \1{\lb\lambda^{-1}_{v'} <\, \lambda^{-1}_{v}\rb} + \1{\lb\lambda^{-1}_{v''} <\, \lambda^{-1}_{v}\rb} + \1{\lb\lambda^{-1}_{v'''} <\, \lambda^{-1}_{v}\rb} \rp -3.$$

For $0\le k \le 4g-3$, we let
$$d(\lambda, k) \de \sum_{i=k}^{4g-3} c(\lambda, \lambda_i).$$

Let $\e\in \ori E(\s)$ be a half-edge and $i$ (resp. $j$) be the smaller (resp. larger) of $\lambda^{-1}_{\e^-}$ and $\lambda^{-1}_{\e^+}$. Then $|l^\e| = l^{\lambda_j} - l^{\lambda_i}$ and $\e$ will contribute to the sum by a factor $+1$ for $l^{\lambda_j}$ and $-1$ for $l^{\lambda_i}$. So $\e$ will contribute to $d(\lambda,k)$ by a factor $+1$ for $k \le j$ plus a factor $-1$ for $k\le i$. Thus the definition we just gave for $d(\lambda,k)$ is consistent with \eqref{dlk}. This, by the way, also prove that $d(\lambda,k)>0$ for $k\neq0$.

\bigskip

We have
$$\sum_{\e\in \ori{E}(\s)} |l^\e| = \sum_{v\in V(\s)} c(\lambda,v)\, l^v = \sum_{i=1}^{4g-3} d(\lambda,i) \lp l^{\lambda_i}-l^{\lambda_{i-1}}\rp.$$
Let us call $k=\lambda^{-1}_{\e_*^-}$. We will write $\1{\lambda} \de \1{\{ l^{\lambda_0} < l^{\lambda_1} < \dots < l^{\lambda_{4g-3}}\}}$ for short. We integrate
$$\1{\lambda}\ \ga^{[6g-1]}\lp \sum_{i=1}^{4g-3} d(\lambda,i) \lp l^{\lambda_i}-l^{\lambda_{i-1}}\rp \rp$$
with respect to $l^{\lambda_{4g-3}}\!$, then $l^{\lambda_{4g-4}}\!$, and so on up to $l^{\lambda_{k+1}}\!$. We then integrate with respect to $l^{\lambda_0}\!$, $l^{\lambda_1}\!$, \dots, $l^{\lambda_{k-1}}\!$. By doing so, factors $\lp d(\lambda,4g-3) \rp^{-1}\!$, $\lp d(\lambda,4g-4) \rp^{-1}\!$, \dots, $\lp d(\lambda,k+1) \rp^{-1}$ then $\lp d(\lambda,1) \rp^{-1}\!$, $\lp d(\lambda,2) \rp^{-1}\!$, \dots, $\lp d(\lambda,k) \rp^{-1}$ successively appear and every time we integrate, $\ga^{[n]}$ is changed into $\ga^{[n+1]}$. All in all,
$$\int d(l^v)_{v\neq \e_*^-} \ \1{\lambda}\  \ga^{[6g-1]}\lp \textstyle\sum_{\e\in \ori{E}(\s)} |l^\e| \rp 
		= \ga^{[10g-4]}(0) \prod_{i=1}^{4g-3} \frac 1 {d(\lambda,i)}.$$

\medskip

The second part of \eqref{intdl} is a little bit trickier because it distinguishes the root from the other vertices. In order to circumvent this, we will consider the sum over all scheme with the same ``unrooted'' structure (we do not consider an ordering $\lambda$ at this time). For any scheme $s\in\Sg$, we note $\unrs$ the non-rooted scheme corresponding to $\s$, and for any non-rooted scheme $\u$, we note~$\u_\e$ the scheme $\u$ rooted at the half-edge $\e$.

Let $\u$ be a non-rooted scheme. We look at $\sum_{\s,\, \unrs =\u} \psi(\s)$ where
$$\psi(\s) \de \int d(l^v)_{v\neq \e_*^-} \ |l^{\e_*}|\, \ga^{[6g-2]}\lp \textstyle\sum_{\e\in \ori{E}(\s)} |l^\e| \rp.$$
This is
\begin{align*}
\sum_{\s,\, \unrs =\u} \psi(\s) &= \frac 1 {\text{Aut}(\u)} \sum_{\e\in\vec E(\u)} \psi(\u_\e)\\
				&= \frac 1 {\lm \{ \s,\, \unrs=\u \} \rm} \sum_{\s,\, \unrs =\u} \frac 1 {\text{Aut}(\u)} \sum_{\e\in \vec E(\u)} \psi(\u_\e)\\
				&= \sum_{\s,\, \unrs =\u} \frac 1 {6g-3} \sum_{\e\in \vec E(\u)} \psi(\u_\e).
\end{align*}

We chose the convention to fix $l^{\e_*^-}$ to be $0$ because we needed one of the $l^v$'s to be $0$ and $\e_*^-$ was already distinguished as the root. This choice was totally arbitrary and we could have taken any other vertex $v_0$. This translates in the fact that, for any function $\chi$,
$$\int d(l^v)_{v\neq \e_*^-}\ \chi\lp (l^\e)_{\e\in \ori{E}(\s)}\rp$$
does not actually depend on $\e_*^-$. In order to see this properly, we do the following change of variables:
$$\text{for every } v \notin \{v_0,\e^-_*\}, \ \tilde l^v \de l^v- l^{v_0}, \hspace{5mm} \tilde l^{\e_*^-} \de - l^{v_0} \sand \tilde l^{v_0} \de 0,$$
so that $\tilde l^\e = l^\e$; and
$$\int d(l^v)_{v\neq \e_*^-}\ \chi\lp (l^\e)_{\e\in \ori{E}(\s)} \rp = \int d(l^v)_{v\neq v_0}\ \chi\lp (l^\e)_{\e\in \ori{E}(\s)}\rp.$$
Using this fact, we see that
$$\psi(\u_\e) = \int d(l^v)_{v\neq \e^-} \ |l^{\e}|\, \ga^{[6g-2]}\lp \textstyle\sum_{\e'\in \ori{E}(\s)} |l^{\e'}| \rp = \int d(l^v)_{v\neq \e_*^-} \ |l^{\e}|\, \ga^{[6g-2]}\lp \textstyle\sum_{\e'\in \ori{E}(\s)} |l^{\e'}| \rp,$$
and
$$\sum_{\s,\, \unrs =\u} \psi(\s) = \sum_{\s,\, \unrs =\u} \frac 1 {6g-3} \int d(l^v)_{v\neq \e_*^-} \lp \textstyle\sum_{\e\in \ori{E}(\s)} |l^\e| \rp \, \ga^{[6g-2]}\lp \textstyle\sum_{\e\in \ori{E}(\s)} |l^\e| \rp.$$

\bigskip

We now consider an ordering $\lambda\in \O_\s$. A computation very similar to the one we conducted above (just change $\ga^{[6g-1]}$ into $x\mapsto x\, \ga^{[6g-2]}(x)$, which becomes, after $4g-3$ successive integrations, $x\mapsto x\,\ga^{[10g-5]}(x) +(4g-3)\ga^{[10g-4]}(x)$) yields
$$\frac 1 {6g-3} \int d(l^v)_{v\neq \e_*^-} \, \1{\lambda} \sum_{\e\in \ori{E}(\s)} |l^\e| \  \ga^{[6g-2]}\lp \textstyle\sum_{\e\in \ori{E}(\s)} |l^\e| \rp
		= \frac{4g-3}{6g-3} \, \ga^{[10g-4]}(0) \prod_{i=1}^{4g-3} \frac 1 {d(\lambda,i)}.$$

The sum over all dominant schemes of \eqref{intdl} then becomes
$$\frac{2(5g-3)}{6g-3}\, \ga^{[10g-4]}(0) \sum_{\s\in\Sg}\sum_{\lambda\in \O_\s} \prod_{i=1}^{4g-3} \frac 1 {d(\lambda,i)}.$$

\subsection{Conclusion}

We still have to compute $\ga^{[10g-4]}(0)$. For that matter, we may use Fubini-Tonnelli's theorem and rewrite \eqref{pn}, for $n\ge 4$, as
\begin{align*}
\ga^{[n]}(0) &= \int_0^\infty \hspace{-3mm} dy_1 \int_0^{y_1} \hspace{-3mm} dy_{2} \dots \int_0^{y_{n-2}} \hspace{-3mm} dy_{n-1} \ \ga_1(y_1) 
	     = \int_0^\infty \hspace{-3mm} dy_1 \ \frac{y_1^{n-2}}{(n-2)!} \ \ga_1(y_1)\\
	     &= \frac 1 {n-2} \int_0^\infty \hspace{-3mm} dy_1 \ \frac{y_1^{n-4}}{(n-4)!} \ \ga_1(y_1)
	     = \frac 1 {n-2} \, \ga^{[n-2]}(0),
\end{align*}
where the second line is obtained from an integration by parts (we differentiate $y\mapsto y^{n-3}$ and integrate $y\mapsto y\,\ga_1(y)$). As $\ga^{[2]}(0)= \frac 12$, we find that
$$\ga^{[10g-4]}(0) = \lp 2^{5g-2} (5g-3)! \rp^{-1}.$$

\bigskip

Taking into account everything we have done so far, we find
$$t_g = \frac 1 {\sqrt{\pi}} \frac{ 3^g\, \Gamma \lp \frac{5g}2 -1 \rp}{2^{6g-3} \, (6g-3)\, (5g-2)!}\, \sum_{\s\in\Sg^*} \sum_{\lambda\in \O_\s}  \prod_{i=1}^{4g-3} \frac 1 {d(\lambda,i)}.$$
The expression we claimed in \eqref{tg} is then obtained by using the identity
$$ \Gamma \lp \frac{5g}2 -1 \rp \Gamma \lp \frac{5g-3}2 \rp = \frac {(5g-2)!}{2^{5g-4}} \, \sqrt{\pi}.$$

\subsubsection*{Acknowledgments}

The author warmly thanks Gr\'egory Miermont for his guidance and constant support throughout the achievement of this work.
\nocite{*}
\bibliographystyle{plain}
\bibliography{slrqpg}

\end{document}

%% file: slrqpg_head.tex
\usepackage{geometry}         

\usepackage{amssymb}
\usepackage{amsmath,cases}
\usepackage{enumerate}
\usepackage{graphicx}
\usepackage{textcomp}
\usepackage{mathrsfs}
\usepackage{bbm}
\usepackage{bm}
\usepackage{caption}
\usepackage{chemarrow}
\usepackage{stmaryrd}
\usepackage{url}
\usepackage{psfrag}
\usepackage[active]{srcltx}   
\usepackage{color}

\renewcommand{\iff}{if and only if }
\renewcommand{\bf}{\textbf}

\newcommand{\ton}{\xrightarrow[n\to\infty]{}}
\newcommand{\tok}{\xrightarrow[k\to\infty]{}}
\newcommand{\tol}{\xrightarrow[n \to \infty]{(d)}}
\newcommand{\tolk}{\xrightarrow[k \to \infty]{(d)}}

\newcommand{\tode}[3]{\xrightarrow[#2 \to #3]{#1}}

\newcommand{\N}{\mathbb{N}}
\newcommand{\Z}{\mathbb{Z}}
\newcommand{\Zp}{\Z_+}
\newcommand{\R}{\mathbb{R}}
\newcommand{\rR}{\mathcal{R}}
\newcommand{\C}{\mathcal{C}}
\renewcommand{\O}{\mathcal{O}}
\newcommand{\U}{\mathcal{U}}
\renewcommand{\S}{\mathcal{S}}
\renewcommand{\P}{\mathbb{P}}
\newcommand{\var}{\text{Var}}
\newcommand{\E}[1]{\mathbb{E}\left[ #1 \right]}

\renewcommand{\L}{\mathfrak{L}}

\newcommand{\K}{\mathcal{K}}
\newcommand{\W}{\mathcal{W}}

\newcommand{\sand}{\text{\hspace{5mm}and\hspace{5mm}}}

\newcommand{\scomp}{,\hspace{8mm}}
\newcommand{\scom}{,\hspace{1cm}}
\newcommand{\cov}{\text{cov}}
\newcommand{\law}{\stackrel{(law)}{=}}
\newcommand{\1}[1]{\mathbbm{1}_{#1}}
\newcommand{\ent}[2]{\llbracket #1,#2 \rrbracket}
\newcommand{\de}{\mathrel{\mathop:}=}

\newcommand{\choo}[2]{\begin{pmatrix}#1 \\ #2 \end{pmatrix}}

\newcommand{\bs}{\backslash}

\newcommand{\eps}{\varepsilon}
\newcommand{\ori}{\check}
\newcommand{\unrs}{\s^{\hspace{-1.8mm}\raisebox{0.5mm}{$\scriptscriptstyle\nrightarrow$}}}

\newcommand{\s}{\mathfrak{s}}
\renewcommand{\l}{\mathfrak{l}}
\renewcommand{\t}{\mathfrak{t}}
\renewcommand{\a}{\mathfrak{a}}
\newcommand{\e}{\mathfrak{e}}
\newcommand{\ccc}{\mathfrak{c}}
\renewcommand{\u}{\mathfrak{u}}
\newcommand{\f}{\mathfrak{f}}
\newcommand{\fF}{\mathfrak{F}}
\newcommand{\CC}{\mathfrak{C}}
\newcommand{\q}{\mathfrak{q}}
\newcommand{\m}{\mathfrak{m}}
\newcommand{\F}{\mathscr{F}}

\newcommand{\Sg}{\mathfrak{S}}

\newcommand{\dgh}{d_{GH}}
\newcommand{\T}{\mathcal{T}}

\newcommand{\Q}{\mathcal{Q}}
\newcommand{\Qb}{\mathcal{Q}^\bullet}
\newcommand{\M}{\mathcal{M}}
\newcommand{\mM}{\mathfrak{M}}

\newcommand{\dH}{\textup{dim}_H}
\newcommand{\dGH}{d_{GH}}
\newcommand{\dis}{\text{dis}}
\newcommand{\lin}{\mathscr{L}}
\newcommand{\rr}{\dot}

\renewcommand{\ll}[1]{\underline #1}
\newcommand{\cc}[1]{\overline #1}

\newcommand{\g}{\gamma \, n^{\frac 1 4}}

\newcommand{\ga}{p}
\renewcommand{\k}{{2(6g-3)}}

\newcommand{\lp}{\left(}
\newcommand{\rp}{\right)}
\newcommand{\lb}{\left\{}
\newcommand{\rb}{\right\}}
\newcommand{\lf}{\left\lfloor}
\newcommand{\rf}{\right\rfloor}
\newcommand{\lc}{\left\lceil}
\newcommand{\rc}{\right\rceil}
\newcommand{\lm}{\left|}
\renewcommand{\rm}{\right|}

\renewcommand{\ln}{\left\|}
\newcommand{\rn}{\right\|}

\newcommand{\rno}{\right.}

\definecolor{gris}{gray}{0.7}

\newtheorem{defi}{Definition}
\newtheorem{prop}{Proposition}
\newtheorem{thm}[prop]{Theorem}
\newtheorem{lem}[prop]{Lemma}
\newtheorem{corol}[prop]{Corollary}
\newenvironment{pre}{\paragraph{Proof.}}{\hspace*{\fill} $\square$ \hspace{-2.45mm}  \vspace{3mm}}
\newenvironment{preuv}[1]{\paragraph{Proof #1.}}{\hspace*{\fill} $\square$ \hspace{-2.45mm} \vspace{3mm}}
\newenvironment{proof}{\paragraph{Proof.}}{\vspace{3mm}}
\newcommand{\qed}[2]{\hspace*{\fill} \raisebox{-#1mm}{$\square$} \hspace{-2.45mm}\vspace*{-#2mm}}
\newenvironment{rem}{\paragraph{Remark.}}{\vspace{3mm}}

\newenvironment{sfig}[3]{
	\begin{figure}[ht]
	\centering
	\includegraphics[#3]{#1}
	\caption{#2}
	\label{#1}
	\end{figure}
}

\newenvironment{tfig}[3]{
	\begin{figure}[ht]
	#3
	\centering
	\includegraphics{#1}
	\caption{#2}
	\label{#1}
	\end{figure}
}

\newenvironment{stfig}[4]{
	\begin{figure}[ht]
	#4
	\centering
	\includegraphics[#3]{#1}
	\caption{#2}
	\label{#1}
	\end{figure}
}

\title{Scaling Limits for Random Quadrangulations of Positive Genus}
\author{J\'er\'emie \textsc{Bettinelli}\\
	\small Laboratoire de Math\'ematiques, Universit\'e Paris-Sud 11\\
	\small F-91405 Orsay Cedex\\
	\small\url{jeremie.bettinelli@normalesup.org}\\
	\small\url{http://www.normalesup.org/~bettinel}}

%% file: slrqpg_abstract.tex
\maketitle
\renewcommand{\labelitemi}{$\diamond$}

\begin{abstract}
We discuss scaling limits of large bipartite quadrangulations of positive genus. For a given $g$, we consider, for every $n \ge 1$, a random quadrangulation $\q_n$ uniformly distributed over the set of all rooted bipartite quadrangulations of genus $g$ with $n$ faces. We view it as a metric space by endowing its set of vertices with the graph distance. We show that, as $n$ tends to infinity, this metric space, with distances rescaled by the factor $n^{-1/4}$, converges in distribution, at least along some subsequence, toward a limiting random metric space. This convergence holds in the sense of the Gromov-Hausdorff topology on compact metric spaces. We show that, regardless of the choice of the subsequence, the Hausdorff dimension of the limiting space is almost surely equal to $4$.

Our main tool is a bijection introduced by Chapuy, Marcus, and Schaeffer between the quadrangulations we consider and objects they call well-labeled $g$-trees. An important part of our study consists in determining the scaling limits of the latter.

\paragraph{Key words:} random map, random tree, conditioned process, Gromov topology.

\paragraph{AMS 2000 Subject Classification:} 60F17.

\bigskip
\noindent
Submitted to EJP on February 18, 2010, final version accepted September 19, 2010.
\end{abstract}


%% file: slrqpg.bbl
\begin{thebibliography}{10}

\bibitem{bender91number}
Edward~A. {Bender} and E.~Rodney {Canfield}.
\newblock {The number of rooted maps on an orientable surface}.
\newblock {\em Journal of Combinatorial Theory, Series B}, 53(2):293--299,
  1991.

\bibitem{bertoin03ptf}
Jean {Bertoin}, Lo\"{\i}c {Chaumont}, and Jim {Pitman}.
\newblock {Path transformations of first passage bridges}.
\newblock {\em Elec. Comm. in Probab}, 8, 2003.

\bibitem{billingsley68cpm}
Patrick {Billingsley}.
\newblock {\em {Convergence of probability measures}}.
\newblock John Wiley \& Sons Inc, New York, 1968.

\bibitem{burago01cmg}
Dimitri {Burago}, Youri {Burago}, and Sergei {Ivanov}.
\newblock {\em {A course in metric geometry}}.
\newblock American Mathematical Society Providence, RI, 2001.

\bibitem{chapuy08sum}
Guillaume {Chapuy}.
\newblock {The structure of unicellular maps, and a connection between maps of
  positive genus and planar labelled trees}.
\newblock {\em Probability Theory and Related Fields}, 147(3):415--447, 2010.

\bibitem{chapuy07brm}
Guillaume {Chapuy}, Michel {Marcus}, and Gilles {Schaeffer}.
\newblock {A bijection for rooted maps on orientable surfaces}.
\newblock {\em SIAM Journal on Discrete Mathematics}, 23(3):1587--1611, 2009.

\bibitem{chassaing04rpl}
Philippe {Chassaing} and Gilles {Schaeffer}.
\newblock {Random planar lattices and integrated superBrownian excursion}.
\newblock {\em Probability Theory and Related Fields}, 128(2):161--212, 2004.
\newblock Springer.

\bibitem{cori81planar}
Robert {Cori} and Bernard {Vauquelin}.
\newblock {Planar maps are well labeled trees}.
\newblock {\em Canad. J. Math}, 33(5):1023--1042, 1981.

\bibitem{duquesne02rtl}
Thomas {Duquesne} and Jean-Fran\c{c}ois {Le Gall}.
\newblock {Random trees, Levy processes and spatial branching processes}.
\newblock {\em Asterisque}, 2002.

\bibitem{ethier86mpc}
Stewart~N. {Ethier} and Thomas~G. {Kurtz}.
\newblock {\em {Markov processes: characterization and convergence}}.
\newblock Wiley, 1986.

\bibitem{federer69gmt}
Herbert {Federer}.
\newblock {\em {Geometric measure theory}}.
\newblock Die Grundlehren der mathematischen Wissenschaften, Band 153.
  Springer-Verlag New York Inc., New York, 1969.

\bibitem{fitzsimmons93mbc}
Pat {Fitzsimmons}, Jim {Pitman}, and Marc {Yor}.
\newblock {Markovian Bridges: Construction, Palm Interpretation, and Splicing}.
\newblock In {\em Seminar on Stochastic Processes}, pages 101--130.
  Birkh{\"a}user, 1992.

\bibitem{gao93number}
Zhicheng {Gao}.
\newblock {The number of degree restricted maps on general surfaces}.
\newblock {\em Discrete Math.}, 123:47--63, 1993.

\bibitem{gromov99msr}
Mikhail {Gromov}.
\newblock {\em {Metric Structures for Riemannian and Non-Riemannian Spaces}}.
\newblock Birkh\"auser, 1999.

\bibitem{legall94mbp}
Jean-Fran\c{c}ois {Le Gall}.
\newblock {Master course: Mouvement brownien, processus de branchement et
  superprocessus}.
\newblock {\em \textup{Available on \url{http://www.dma.ens.fr/~legall/}}},
  1994.

\bibitem{legall99sbp}
Jean-Fran\c{c}ois {Le Gall}.
\newblock {\em {Spatial Branching Processes, Random Snakes, and Partial
  Differential Equations}}.
\newblock Birkh{\"a}user, 1999.

\bibitem{legall07tss}
Jean-Fran\c{c}ois {Le Gall}.
\newblock {The topological structure of scaling limits of large planar maps}.
\newblock {\em Inventiones Mathematicae}, 169(3):621--670, 2007.
\newblock Springer.

\bibitem{legall09scaling}
Jean-Fran\c{c}ois {Le Gall} and Gr\'egory {Miermont}.
\newblock {Scaling limits of random planar maps with large faces}.
\newblock {\em \textup{\url{arXiv:0907.3262}}, to appear in Ann. Probab.},
  2009.

\bibitem{legall08slb}
Jean-Fran\c{c}ois {Le Gall} and Fr\'ed\'eric {Paulin}.
\newblock {Scaling limits of bipartite planar maps are homeomorphic to the
  2-sphere}.
\newblock {\em Geometric and Functional Analysis}, 18(3):893--918, 2008.
\newblock Springer.

\bibitem{marckert03sss}
Jean-Fran\c{c}ois {Marckert} and Abdelkader {Mokkadem}.
\newblock {States spaces of the snake and its tour, convergence of the discrete
  snake}.
\newblock {\em Journal of Theoretical Probability}, 16(4):1015--1046, 2003.
\newblock Springer.

\bibitem{marckert06limit}
Jean-Fran\c{c}ois {Marckert} and Abdelkader {Mokkadem}.
\newblock {Limit of normalized quadrangulations: the Brownian map}.
\newblock {\em Annals of Probability}, 34(6):2144--2202, 2006.
\newblock Hayward, Calif. Institute of Mathematical Statistics.

\bibitem{miermont08sphericity}
Gr\'egory {Miermont}.
\newblock {On the sphericity of scaling limits of random planar
  quadrangulations}.
\newblock {\em Electronic Communications in Probability}, 13:248--257, 2008.

\bibitem{miermont09trm}
Gr\'egory {Miermont}.
\newblock {Tessellations of random maps of arbitrary genus}.
\newblock {\em Annales scientifiques de l'\'Ecole Normale Sup\'erieure},
  42(5):725--781, 2009.

\bibitem{neveu86apg}
Jacques {Neveu}.
\newblock {Arbres et processus de Galton-Watson}.
\newblock {\em Ann. Inst. H. Poincare}, 22(2):199--207, 1986.

\bibitem{petrov75sir}
Valentin~V. {Petrov}.
\newblock {Sums of Independent Random Variables.}
\newblock 1975.
\newblock Translated from the Russian by A. A. {Brown}, Ergebnisse der
  Mathematik und ihrer Grenzgebiete, Band 82.

\bibitem{petrov95ltp}
Valentin~V. {Petrov}.
\newblock {\em {Limit theorems of probability theory: sequences of independent
  random variables}}.
\newblock Oxford University Press, USA, 1995.

\bibitem{revuz99cma}
Daniel {Revuz} and Marc {Yor}.
\newblock {\em {Continuous martingales and Brownian motion}}.
\newblock Springer, 1999.

\bibitem{schaeffer98cac}
Gilles {Schaeffer}.
\newblock {\em {Conjugaison d'arbres et cartes combinatoires al\'eatoires}}.
\newblock PhD thesis, Universit\'e de Bordeaux 1, 1998.

\bibitem{schwartz70atg}
Laurent {Schwartz}.
\newblock {\em {Analyse~: Topologie generale et Analyse fonctionnelle}}.
\newblock Hermann Paris, 1970.

\bibitem{stroock99pta}
Daniel~W. {Stroock}.
\newblock {\em {Probability theory, an analytic view}}.
\newblock Cambridge University Press, 1999.

\end{thebibliography}
